\documentclass[11pt]{article}
\usepackage[top=1in, bottom=1in, left=1in, right=1in]{geometry}
\usepackage{srcltx,graphicx}
\usepackage{amsmath, amssymb, amsthm}
\usepackage{color, xcolor, colortbl}
\usepackage{algorithm}
\usepackage{bbm}
\usepackage{algpseudocode}
\usepackage{url}
\usepackage[capitalize,nameinlink]{cleveref} 
\crefname{problem}{Problem}{Problems}
\usepackage[justification=centering]{subfig}

\usepackage{pgfplots}
\pgfplotsset{
  colormap={CM}{rgb255=(255,255,255) rgb255=(201,201,201) },
  tick label style={font=\footnotesize},
  label style={font=\footnotesize},
}

\numberwithin{equation}{section}

\title{Maximizing robustness of point-set registration by leveraging non-convexity}
\date{}
\author{
	   Cindy Orozco Bohorquez\thanks{Institute for Computational and Mathematical Engineering, Stanford University, Stanford, CA 94305. Email: {\tt orozcocc@stanford.edu}}, \and
	Yuehaw Khoo\thanks{Department of Statistics, University of Chicago,Chicago, IL 60637. Email: {\tt ykhoo@uchicago.edu}},\and
			   Lexing Ying\thanks{Department of Mathematics and ICME, Stanford University,
			Stanford, CA 94305.		     Stanford, CA 94305.
			Email: {\tt lexing@stanford.edu}}}

\theoremstyle{plain}
\newtheorem{theorem}{Theorem}[section]
\newtheorem{lemma}{Lemma}[section]
\newtheorem{corollary}{Corollary}[section]

\theoremstyle{definition}
\newtheorem{definition}{Definition}[section]
\newtheorem{problem}{Problem}[section] 
\newcommand{\Unif}[0]{\text{Unif}}  
\newcommand{\E}[0]{\mathbb{E}}  
\newcommand{\minimize}{\mathop{\mathrm{minimize}}\displaylimits}  
\newcommand{\SO}[0]{\mathcal{SO}}  
  
\newcommand{\Tr}[0]{\text{Tr}}  

\begin{document}
\maketitle
\begin{abstract}
Point-set registration is a classical image processing problem that looks for the optimal
transformation between two sets of points. In this work, we analyze the impact of outliers when
finding the optimal rotation between two point clouds. The presence of outliers motivates the use of
least unsquared deviation, which is a non-smooth minimization problem over non-convex domain. We
compare approaches based on non-convex optimization over special orthogonal group and convex
relaxations. We show that if the fraction of outliers is larger than a certain threshold, any naive
convex relaxation fails to recover the ground truth rotation regardless of the sample size and
dimension. In contrast, minimizing the least unsquared deviation directly over the special
orthogonal group exactly recovers the ground truth rotation for any level of corruption as long as
the sample size is large enough. These theoretical findings are supported by numerical simulations.
\end{abstract}

\section{Introduction}\label{sec:intro}


In this paper, we study the problem of aligning two point clouds where one of the point clouds is
subjected to gross corruption. More precisely, given a ground truth rotation $R_0\in\SO (d)$ and a
set of indices $\mathcal{C}\subseteq\{1,\dots,N\}$, we model the two point clouds $\{x_i\}_{i=1}^N,
\{y_i\}_{i=1}^N \subset \mathbb{R}^d$, where
\begin{align}
  x_i \sim \Unif(\mathbb{S}^{d-1})\ \text{for} \ 1 \leq i \leq N\label{eq:prob_model0}
\end{align}
independently, and
\begin{align}
  \begin{cases}
    y_i = R_0 x_i\ & \text{for } i \in \mathcal{C}^{c}\\ y_i \sim \Unif(\mathbb{S}^{d-1})& \text{for } i
    \in \mathcal{C}\ \text{independently}. \\
  \end{cases}\label{eq:prob_model}
\end{align}
Here $\Unif(\mathbb{S}^{d-1})$ denotes uniform distribution over $\mathbb{S}^{d-1}$, the unit sphere
in dimension $d$, $\mathcal{C}^{c}$ denotes the complement of $\mathcal{C}$,
$\mathcal{C}^{c}:=\{1,\dots,N\}\setminus\mathcal{C}$, and $\SO(d)$ denotes the special orthogonal
group
\begin{equation}
\SO(d) = \{R\in\mathbb{R}^{d\times d} \ \vert \ R^T R = I_d, \ \det{R} = 1\}.
\end{equation}
The model \eqref{eq:prob_model} implies the points $x_i$ for $i \in \mathcal{C}^c$ can be aligned
with $y_i$ from the same index set via applying some ground truth rotation $R_0$, while for $i \in
\mathcal{C}$, $y_i$ and $x_i$ are generated independent of each other (hence cannot be
aligned). Therefore $\mathcal{C}$ is the the index set of corrupted points, and we denote the
corruption level as
\begin{align}
p:=\frac{\vert \mathcal{C} \vert}{N}.
\end{align}
When $p=1$, the points in $\{y_i\}_{i=1}^N$ are all corrupted, while there is no corruption when
$p=0$. The goal of point-set registration is to recover the ground truth rotation $R_0$ given the
point clouds $\{x_i\}_{i=1}^N, \{y_i\}_{i=1}^N$ when the index set of corrupted points,
$\mathcal{C}$, is unknown, and the corruption level $p$ can take any value in $[0,1)$.

In order to limit the influence of outliers when determining the rotation, we minimize the least
unsquared deviation (LUD) \cite{wang_2013} defined to be
\begin{align}
  \label{eq:LUDdef}
  L(A; \{x_i\}_{i=1}^N, \{y_i\}_{i=1}^N) := \frac{1}{N}\sum_{i=1}^N \| A x_i - y_i\|_2,\quad
  A\in\mathbb{R}^{n\times n}.
\end{align}
Therefore it is natural to determine the ground truth rotation $R_0$ via minimizing the LUD. 
\begin{problem}\label{prob:LUD} 
  $\minimize L(R; \{x_i\}_{i=1}^N, \{y_i\}_{i=1}^N)$ such that $R\in \SO(d)$.
\end{problem}
This problem is however, non-convex, due to the domain $\SO(d)$.

\subsection{Previous approaches}

In $\mathbb{R}^{3}$, the point-registration problem without outliers can be formulated as the
Wahba's problem \cite{wahba_65},\cite{wahba2_66}, where the rotation $R_{0}$ is recovered via
solving the least squares (LS)
\begin{align}\label{eq:LS}
\min_{R\in \SO(d)} LS(R);\quad 
LS(R; \{x_i\}_{i=1}^N, \{y_i\}_{i=1}^N) := \frac{1}{N}\sum_{i=1}^N \| R  x_i - y_i\|_2^{2}.
\end{align}

The Wahba's problem is equivalent to the orthogonal Procrustes problem \cite{Schonemann_1966}, with
the additional constraint that the determinant of the solution equals to $+1$.  The solution of
\eqref{eq:LS} can be computed using the singular value decomposition (SVD)
\cite{Ruiter_2013},\cite{Markley_87}. Given that the solution of \eqref{eq:LS} is sensitive to
outliers, in computer vision \eqref{eq:LS} is reformulated as a Maximum Consensus problem. The most
common algorithm to solve the maximum consensus problem is random sample consensus (RANSAC)
\cite{RANSAC}. A survey of approximate and exact algorithms to solve the maximum consensus problem
can be found in \cite{Chin_2017}.

Another approach to deal with outliers is to change the loss function in \eqref{eq:LS} by Huber loss
\cite{Verboon1992}, LUD \cite{wang_2013} or truncated-least-squares deviation \cite{Yang_2019}. The
LUD has been proven to be more robust to outliers than \eqref{eq:LS} specifically in the context of
robust registration \cite{wang_2013}, camera location recovery \cite{Lerman_2018} and robust
subspace recovery \cite{Lerman_2019}. Observing the LUD is a convex function, one can obtain a
convex problem from \cref{prob:LUD} via applying a convex relaxation to the domain $\SO(d)$. Several
different kinds of relaxation can be applied. As a baseline, one can consider solving the
unconstrained problem:
\begin{problem}\label{prob:conv1}
  $\minimize L(A; \{x_i\}_{i=1}^N, \{y_i\}_{i=1}^N)$ such that $A\in\mathbb{R}^{d\times d}$.
\end{problem}
Compared with \cref{prob:LUD}, here the
rotation group constraint is removed. Further improvement can be obtained if we solve:
\begin{problem}\label{prob:conv2}
	$\minimize L(A; \{x_i\}_{i=1}^N, \{y_i\}_{i=1}^N)$ such that $A\in\mathrm{conv}\, \SO(d)$.
\end{problem}
Here the domain is relaxed to the convex hull of $\SO(d)$, $\mathrm{conv}\, \SO(d)$, which admits
characterization via positive semidefinite matrices of with size $2^{d-1}\times 2^{d-1}$
(exponential in $d$) \cite{Saunderson_2015}. This approach is used in \cite{Horowitz_14} when
optimizing over the convex hull of the special Euclidean group for two and three dimensions.
However, this approach assumes that the proportion of outliers is rather small. Similarly, in
\cite{Yang_2019}, a tight semidefinite program relaxation using quaternions is proposed to minimize
the truncated least squares deviation in $d=3$. The dimension of the domain of the semidefinite
relaxation depends quadratically in the sample size $N$.

On the other hand, recent years have seen many instances of non-convex optimization problems that
have a rather benign optimization landscape. In these examples, either all critical points are
saddle points and global optima, or the basin of convergence is large \cite{Sun}. Therefore, a
direct minimization via first order method, sometimes with the help of a cheap initialization, can
solve the non-convex problem with optimality guarantees. Such benign behavior has been seen in a
related problem of robust principal component analysis \cite{Lerman_2018,Lerman_2019}. As for our
problem, in \cite{Verboon1992}, an iterative reweighted least squares is proposed to solve
\cref{prob:LUD} with rather encouraging results. Thus we ask the natural question of whether
\cref{prob:LUD} admits a benign optimization landscape that allows the use of cheap first-order method
instead of a convex relaxation with exponential complexity in $d$.

\subsection{Our contributions}
The contributions of this paper are two-fold:
\begin{itemize}
\item Although convex relaxation yields a surrogate optimization problem where the global optima can
  always be achieved, when $p$ is large it is possible that the solutions of \cref{prob:conv1} and
  \cref{prob:conv2} do not coincide with $R_0$. In this paper, we prove when $p$ is sufficiently large
  ($p\approx 0.6$), optimizing LUD over any convex set that contains $\SO(d)$ does not recover the
  ground truth rotation $R_0$.
\item Motivated by such observation, we solve \cref{prob:LUD} using non-convex optimization to
  explicitly constrain the solution to be in $\SO(d)$. We prove that by minimizing LUD in
  $\SO(d)$ starting at almost any initialization point, one can always recover $R_0$ for any $p<1$,
  as long as the sample size $N$ is sufficiently large. This is yet another example where the use of
  non-convex optimization is superior than convex relaxation approaches.  phenomenon.
\end{itemize}

 \begin{figure}[htb]
 	\centering
 	\includegraphics{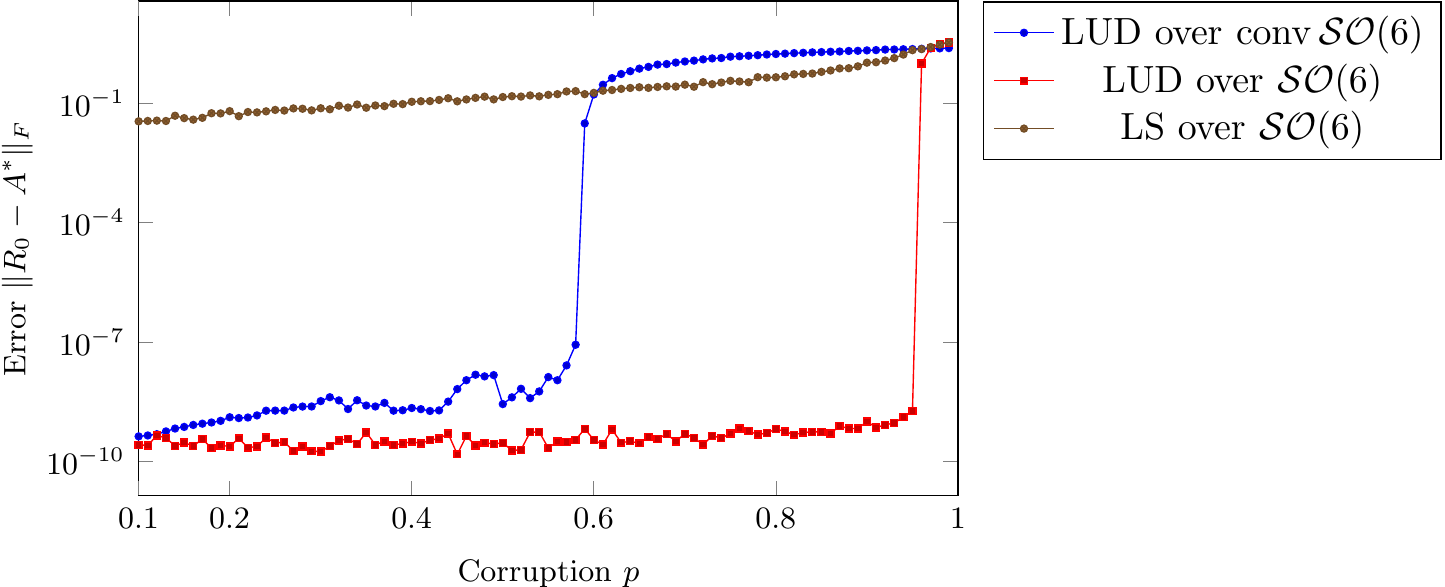}
 	\caption{Error for different methods. $A^{*}$ denotes the output of certain minimization
      procedure. Red: Solving \cref{prob:LUD}. Blue: Solving \cref{prob:conv2}. Brown: Solving
      the least-squares problem in \eqref{eq:LS}. }\label{fig:Puzzle}
 \end{figure}

A summary of our contributions can be found in \cref{fig:Puzzle}, where we compare the error for
minimizing LS in \eqref{eq:LS} and solving \cref{prob:LUD} and
\cref{prob:conv2} for dimension $d = 6$ and sample size $N=1024$. Details of numerical implementation
are provided in \cref{sec:Num.Sim.}. Although it can be shown that minimizing LS asymptotically
ensures $A^*\rightarrow R_0$ as $N\rightarrow\infty$, for finite sample size LS is sensitive to
outliers. In \cref{fig:Puzzle}, the empirical error between the ground truth rotation and the
solution of LS is larger than $10^{-2}$, even when the corruption level is as low as $p=0.1$. On the
contrary, minimizing LUD over either $\SO(6)$ or $\mathrm{conv}\,\SO(6)$ exactly recovers the ground
truth when the corruption level is small. The maximum admissible corruption level to recover the
ground truth is 0.6 if we optimize LUD over $\mathrm{conv}\, \SO(6)$, and $0.95$ if we optimize over
$\SO(6)$. More interestingly, these thresholds seem to be independent of the dimension.

\subsection{Summary of results}
All universal constants are denoted by the notation $c$, unless in situations where we explicitly
distinguish between different constants. Moreover we say an event happens \emph{with high
  probability}, when the probability of the event happening is larger than $1-c/N$ for some
universal constant $c$.

We start by characterizing the conditions under which the solution of convex relaxations of
\cref{prob:LUD} recovers the ground truth rotation. In particular we focus on convex relaxations of
the optimization domain $\SO(d)$, $\mathbb{R}^{d\times d}$, and $\mathrm{conv}\,\SO(d)$,
i.e. \cref{prob:conv1} and \cref{prob:conv2}, respectively.  We define the admissible corruption threshold 
\begin{align}
  \label{eq: definition p tilde}
  \tilde{p}(d) :=
  \left(1+\frac{B\left(d-1,\frac{1}{2}\right)}{B\left(\frac{d-1}{2},\frac{1}{2}\right)}\right)^{-1},
\end{align}
where $B(m,n)$ is the beta function
$B\left(m,n\right):=2\int_{0}^{\pi/2}\sin^{2n-1}\theta\cos^{2m-1}\theta d\theta$. $\tilde{p}(d)$
is a decreasing function of $d$, such that $\tilde{p}(3)=0.6$ and
$\tilde{p}(d)\rightarrow(1+1/\sqrt{2})^{-1}\approx0.5858$ as $d\rightarrow \infty$. We first
establish when it is possible to recover the ground truth rotation via the baseline approach
\cref{prob:conv1} in \cref{thm:sharpness convex}.
\begin{theorem}[Exact recovery using unconstrained optimization, \cref{prob:conv1}]\label{thm:sharpness convex}
	Given data $\{x_i\}_{i=1}^{N}$, $\{y_i\}_{i=1}^{N}$ sampled from the distribution
    \eqref{eq:prob_model0} and \eqref{eq:prob_model}, if the corruption level
	\begin{equation}
	\label{eq:possible_recovery} 
	p<\tilde{p}(d) - O\left(d\sqrt\frac{\log N}{N}\right),
	\end{equation}
then, the ground truth rotation $R_0$ is the unique minimizer of
$L(A;\{x_i\}_{i=1}^{N},\{y_i\}_{i=1}^{N})$ over $\mathbb{R}^{d\times d}$, with high probability.
\end{theorem}

One may wonder whether we can expect recovery for a larger corruption level $p$ than the bound in
\eqref{eq:possible_recovery} using a tighter relaxation than \cref{prob:conv1},
i.e. \cref{prob:conv2}. We show in \cref{thm:imp. recovery} that indeed the condition
\eqref{eq:possible_recovery} is necessary, for the recovery of ground truth rotation via any convex
relaxation.

\begin{theorem}[Recovery failure optimizing over $\mathrm{conv}\,\SO(d)$, \cref{prob:conv2}]\label{thm:imp. recovery}
  Given data $\{x_i\}_{i=1}^{N}$, $\{y_i\}_{i=1}^{N}$ sampled from the distribution
  \eqref{eq:prob_model0} and \eqref{eq:prob_model}, if the corruption level
  \begin{align}
	\label{eq:impossible_recovery} 
	p > \tilde{p}(d) + O\left(\sqrt{\frac{d\log N}{N}}\right),
  \end{align}
  then with high probability the ground truth rotation $R_0$ is not a minimizer of
  $L(A;\{x_i\}_{i=1}^{N},\{y_i\}_{i=1}^{N})$ over $\mathrm{conv}\, \SO(d)$. Moreover, for any convex
  set $\mathcal{P}\supseteq \SO(d)$, let $A_{\mathcal{P}}^*$ be a minimizer of
  $L(A;\{x_i\}_{i=1}^{N},\{y_i\}_{i=1}^{N})$ over $\mathcal{P}$, then
  \begin{align}
	\label{eq:impossible_recovery_distance} 
	\|A_{\mathcal{P}}^{*}-R_0\|_{F} > O\left(p- \tilde{p}(d) \right).
  \end{align}
\end{theorem}

Finally, we characterize the efficacy of solving \cref{prob:LUD} for recovering the ground truth
$R_0$. To this end, we identify the conditions under which a gradient descent on the manifold
$\SO(d)$ converges to $R_0$. More precisely, we show that a dynamical system on $\SO(d)$ following a
minimizing gradient flow of the LUD cost converges in finite time to the ground truth rotation given
some mild assumptions on the initial point.

\begin{theorem}[Exact recovery in finite time optimizing over $\SO(d)$, \cref{prob:LUD}]
  \label{thm:converg. nonconvex}
  Let $\{x_i\}_{i=1}^{N},\{y_i\}_{i=1}^{N}$ be sampled from the distribution in
  \eqref{eq:prob_model0} and \eqref{eq:prob_model}, $\partial_{R}L(R)$ be the Riemannian
  generalized gradient of $L(R;\{x_i\}_{i=1}^{N},\{y_i\}_{i=1}^{N})$ at $R\in\SO(d)$. With high
  probability, all solutions of the dynamical system
  \begin{align}
    \frac{d R}{dt}(t) \in - \partial_{R}L(R),\label{eq:dyn_sysLUD}
  \end{align}
  with initial condition $R(0)$ such that
  \begin{align}\label{eq: in_R_dyn}
 	\|\log (R_{0}^\top R(0))\|_{2}<\pi-O\left(\left(\frac{d^3}{(1-p)^2}\frac{(\log(N)+d)}{N}\right)^{1/4}\right)
  \end{align}
  converge to $R_0$ in finite time, i.e. $R(t)=R_0$ for all $t\geq T(\|\log(R_{0}^\top R(0))\|_{2})$. Here
  \begin{align}
    \label{eq:conv rate}
    T(s) := \frac{d}{1-p}\left(\cosh^{-1}\left(\sec\left(\frac{s}{2}\right)\right)
    +c\left(\frac{d^3}{(1-p)^2}\frac{(\log(N)+d)}{N}\right)^{1/4}\right),\ s\in[0,\pi).
  \end{align}
\end{theorem} 

In \cref{thm:converg. nonconvex}, the dynamical system \eqref{eq:dyn_sysLUD} is a generalization of
the differential equation when the gradient is nonsmooth. Additionally, condition \eqref{eq:
  in_R_dyn} is equivalent to
\begin{align}
\label{eq: cond_p_non_cvx}
p < 1 - O\left(\frac{d}{(\pi-\|\log(R_{0}^\top R(0))\|_{2})^4}\sqrt{\frac{d(\log(N)+d)}{N}}\right).
\end{align}
Therefore, the maximum corruption ratio to ensure exact recovery solving \cref{prob:LUD} in finite time goes
to one as the sample size goes to infinity.

In the rest of the paper where there is no ambiguity, we simply write $L(M; \{x_i\}_{i=1}^N,
\{y_i\}_{i=1}^N)$ as $L(M)$, though it should be understood that $L(M)$ depends on the random
variables $\{x_i\}_{i=1}^N, \{y_i\}_{i=1}^N$. Because of the rotation invariance of the data
uniformly distributed on $\mathbb{S}^{d-1}$, without loss of generality, we also fix the ground
truth rotation $R_0=I$.

\subsection{Organization}
In \cref{sec:pre}, we present basic notations, definitions and theorems required for the rest of the
paper.  In \cref{sec: behavior cvx}, we study the properties of \cref{prob:conv1} and \cref{prob:conv2}
and prove \cref{thm:sharpness convex} and \cref{thm:imp. recovery}. In \cref{sec:fin time SOd}, we
tailor some results in dynamical system theory in preparation for proving
\cref{thm:converg. nonconvex}. In \cref{sec:non-convex recovery}, we show that with a proper
initialization, there is always exact recovery for any $p<1$ when considering the noise model
\eqref{eq:prob_model}, therefore proving \cref{thm:converg. nonconvex}. In \cref{sec:Num.Sim.}, we
support our theoretical findings via numerical simulations.

\section{Preliminaries}\label{sec:pre}
In this section, we introduce some background in concentration inequalities, dynamical systems,
Riemannian manifolds, and $\varepsilon$-nets. Additional background in random variables uniformly
distributed on $\mathbb{S}^{d-1}$ is given in \cref{sec:prop_sphere}.

\subsection{Concentration inequalities}

Before introducing the concentration inequalities used in this paper, we first introduce the
\textit{sub-gaussian} norm of a random variable $X$:
\begin{equation}
\| X \|_{\psi_2} := \inf_t \left\{t\ \bigg\vert \ \mathbb{E}\exp\left(\frac{X^2}{t^2}\right)\leq 2\right \}.
\end{equation}
Intuitively, this norm measures the spread of the distribution of a random variable $X$. If the
distribution of $X$ has a large tail, then $\| X \|_{\psi_2}$ has to be large to ensure
$\mathbb{E}\exp(X^2/ \|X\|_{\psi_2}^2)\leq 2$. Indeed, for a Gaussian variable $X$ with variance
$\sigma^2$, $\|X\|_{\psi_2}^2 = c \sigma^2$ for some constant $c>0$.

If $\|X\|_{\psi_2}$ is bounded, we say that $X$ is a sub-gaussian random variable.  In addition if
$\E X=0$, $X$ being sub-gaussian is equivalent to
\begin{equation}
\label{eq: MGF subgaussian}
\E\exp(\lambda X)\leq \exp(c\lambda^2\|X\|^2_{\psi_2})
\end{equation}
for some universal constant $c>0$.

For a sub-gaussian random variable, the following theorem holds.
\begin{theorem}[Hoeffding's inequality, \cite{vershynin_2018}] \label{theorem:hoeffding}
	Let $X_1,\ldots, X_M$ be independent sub-gaussian random variables. Then for every $t\geq 0$ 
	\begin{equation}
	\text{P}\left\{\left|\sum_{i=1}^M X_i-\E X_i\right| \geq t\right\} \leq \exp\left(
    \frac{-ct^2}{\sum_{i=1}^M \| X_i\|^2_{\psi_2}}\right)
	\end{equation}
	for some universal constant $c>0$.
\end{theorem}

We next introduce Talagrand's inequality. First we need a few definitions.  We use the \emph{radius}
and \emph{gaussian width} to measure the size of a set $T\in \mathbb{R}^d$:
\begin{align}
  \label{eq:gauss width rad def}
  \text{rad}(T) &= \sup_{t\in T} \|t\|_2,\nonumber\\ w(T) &= \mathbb{E} \sup_{t\in T}\ \langle g, t
  \rangle,
\end{align}
where $g\in\mathbb{R}^n$ is a random vector with $\mathcal{N}(0,1)$ independently distributed
entries. With these definitions, we introduce the following theorem.
\begin{theorem}[Talagrand's comparison inequality in tail bound form, \cite{vershynin_2018}]\label{theorem:talagrand}
  Let $T\subset \mathbb{R}^{n}$ be a set and $\{X_t\}_{t\in T}$ be a random process indexed by
  elements in $T$, such that $X_{0}=0$. If for all $t,s\in T\cup\{0\}$
  \begin{equation}\label{eq:cond_tal}
	\| X_t - X_s \|_{\psi_2}\leq K\|t-s\|_2,
  \end{equation}
  then for some universal constant $c>0$,
  \begin{equation}
	\text{P}\left\{\sup_{t\in T} \vert X_t \vert \leq cK(w(T)+\mathrm{rad}(T)\ u)\right\}\geq 1-2\exp(-u^2).
  \end{equation}
\end{theorem}

\subsection{Discontinuous dynamical systems}
The differential equation \eqref{eq:dyn_sysLUD} in \cref{thm:converg. nonconvex} is called a
differential inclusion, a term that generalizes dynamical systems when the forcing term is a set of
functions. An introductory summary to the area can be found in \cite{Cortes_08}, as well as
classical books \cite{Aubin_84} and \cite{Filippov_88}. In \cref{sec:fin time SOd}, we derive finite
time convergence for a particular set of differential inclusions in $\SO(d)$. This results are based
in generalizations of Lyapunov functions for differential inclusions in
\cite{Baccioti_99}. Numerical methods to solve differential inclusions can be found in
\cite{Acary_08}.
\begin{definition}[Differential inclusion, \cite{Acary_08}]
  A differential inclusion is defined by
	\begin{align}
	\label{eq: DI}
	\frac{dz}{dt}\in \mathcal{F}(z(t)),\quad t\in[0,T],\quad  z(0) = z_0,
	\end{align}
	where $z:\mathbb{R}\rightarrow\mathcal{D}\subseteq\mathbb{R}^{n}$ is a function of time,
    $\mathfrak{B}(\mathbb{R}^{n})$ is the collection of subsets of $\mathbb{R}^{d}$,
    $\mathcal{F}:\mathcal{D}\rightarrow\mathfrak{B}(\mathbb{R}^{n})$ is a set-valued map which
    associates to any point $z\in \mathcal{D}$ with a set $\mathcal{F}(z)\subset\mathbb{R}^{n}$ and
    $T>0$.
\end{definition}
\begin{definition}[Solution of differential inclusion]
  A solution $z(t)$ of the differential inclusion \eqref{eq: DI} is absolutely continuous, i.e. it
  satisfies
  \begin{align}
	z(t) = z(0) + \int_{0}^{t}\frac{dz(s)}{ds}\, ds
  \end{align}
  such that $\frac{dz}{dt}(t)\in\mathcal{F}(z(t))$ almost everywhere.
\end{definition}
\begin{definition}[Generalized gradient, \cite{Baccioti_99}]\label{def:def.gen.grad}
  Let $f:\mathbb{R}^{n}\rightarrow\mathbb{R}$ be a locally Lipschitz function, then the generalized
  gradient of $f$ at $z$ is the set
  \begin{align}
	\bar\partial f(z):=\mathrm{conv}\left\{\lim_{j\rightarrow+\infty}\nabla f(z_{j})\mid
    z_{j}\rightarrow z, z_{j}\notin G, z_{j}\notin G_{f}\right \}
  \end{align}
  where $G$ is any set of zero measure in $\mathbb{R}^{n}$ and $G_{f}$ is the set of measure zero
  where the Euclidean gradient $\nabla f$ does not exists.
\end{definition}

\subsection{$\SO(d)$ as a Riemannian manifold}\label{sec: riemanian manifold}
$\SO(d)$ can be seen as an embedded submanifold of $\mathbb{R}^{d\times d}$, or as a matrix Lie
group. We follow the manifold optimization approach in \cite{absil_2008} to define Riemannian
gradient over $SO(d)$. Then we look at $\SO(d)$ as a matrix Lie group to expose the close connection
with skew-symmetric matrices of dimension $d$, following \cite{Hall_2015}.

\begin{definition}[Special orthogonal group]
	The special orthogonal group of dimension $d$, $\SO(d)$ is defined as 
	\begin{align}
	\SO(d):=\{X\in\mathbb{R}^{d\times d}\mid X^\top X = I, \det(X) = 1 \}.
	\end{align}
\end{definition}
Notice than the dimension of $\SO(d)$ is $d(d-1)/2$. 

\begin{definition}[Tangent space of $\SO(d)$]
	For $R\in\SO(d)$, the tangent space of $\SO(d)$ at a rotation $R$ is defined as
	\begin{align}
	\label{def: Tangent space}
	T_{R}\SO(d):=R\cdot\mathcal{S}_{\mathrm{skew}}(d)
	\end{align}
	where the set of all skew-symmetric $d\times d$ matrices is defined as 
	\begin{align}
	\mathcal{S}_{\mathrm{skew}}(d):=\{X\in\mathbb{R}^{d\times d}\mid -X=X^\top\}.
	\end{align}
\end{definition}

As an embedded submanifold of $\mathbb{R}^{d\times d}$, the Riemmannian gradient of a function $f$
over $\SO(d)$ is defined as the projection of the Euclidean gradient over $T_{R}\SO(d)$.

\begin{definition}[Riemannian generalized gradient of $f$ over $\SO(d)$]
  Given $f:{\SO(d)}\rightarrow\mathbb{R}$ with locally Lipschitz extension
  $\bar{f}:\mathbb{R}^{d\times d}\rightarrow\mathbb{R}$ such that
  $f=\left.\bar{f}\right\vert_{\SO(d)}$,then the \textit{Riemannian generalized gradient} of $f$ is
  defined as
	\begin{align}
	\label{def: Rieman.gen.grad}
	\partial_{R} f(R) :=R\cdot\mathrm{skew}(R^\top (\bar\partial\bar{f}(R))) 
	\end{align}   
	where $ \mathrm{skew}(A)=(A-A^\top)/2$ and $\bar\partial\bar{f}$ is the Euclidean generalized
    gradient of $\bar{f}$ as defined in \cref{def:def.gen.grad}.
\end{definition}

We introduce the matrix exponential and the principal matrix logarithm.

\begin{definition}[Matrix exponential]
  For any $A\in\mathbb{C}^{d\times d}$, $\exp(A) = \sum_{k=0}^\infty \frac{A^{k}}{k!}$
\end{definition}

\begin{definition}[Principal logarithm of a matrix]
  For any $A\in\mathbb{C}^{d\times d}$ with no negative eigenvalues, there exists a
  unique matrix $C\in\mathbb{C}^{d\times d}$ such that $\exp(C)=A$ and the imaginary part of the
  eigenvalues of $C$ are in the interval $(-\pi,\pi)$. The principal logarithm is defined as $\log(A):= C$
\end{definition}

\begin{definition}[Geodesics on $\SO(d)$]
  For any $R\in\SO(d)$ and $A\in\mathcal{S}_{\mathrm{skew}}(d)$, the geodesic
  $\gamma:\mathbb{R}\rightarrow\SO(d)$ such that $\gamma(0) = R$ and $\dot{\gamma}(0) = R\cdot A\in
  T_{R}\SO(d)$ is given by
  \begin{align}
	\gamma(t) = R\cdot\exp(tA).
	\end{align}
\end{definition}

The matrix exponential map on $\SO(d)$ is surjective \cite[Chapter~11]{Hall_2015}.Therefore, given the set
\begin{align}\label{eq:def_B_skew}
\mathcal{B}_{\mathrm{skew}}(d):=\{S\in\mathcal{S}_\mathrm{skew}(d)\mid \|S\|_{2}\leq 1\}.
\end{align}
then $\SO(d)=\exp\left(\pi\mathcal{B}_{\mathrm{skew}}(d)\right)$.This implies that it is possible to
define a principal logarithm for all rotations with the image contained in $\pi \mathcal{B}_{\mathrm{skew}}(d)$.

  
Additionally, for any pair $R,Q\in\SO(d)$, the geodesic $\gamma(t) = R\cdot \exp(t\log(R^\top Q))$
satisfies $\gamma(0)=R$ and $\gamma(1)=Q$. This remark provides a practical definition of the
Riemannian distance in $\SO(d)$.

\begin{definition}[Riemannian distance in $\SO(d)$]
  Let $R,Q\in\SO(d)$ then the Riemannian distance in $\SO(d)$ is given by
  \begin{align}\label{eq:riemannian_distance}
	D_{\SO}(R,Q) := \|\log(Q^\top R)\|_{F}.
  \end{align}
\end{definition}

\subsection{ $\varepsilon$-nets}\label{sec:enet}
A discrete set $\mathcal{N}^{\varepsilon}_{\mathcal{P}}$ is called an $\varepsilon$-net for a set
$\mathcal{P}$ if for any $a\in \mathcal{P}$, there exists $b\in
\mathcal{N}^{\varepsilon}_{\mathcal{P}}$ such that, for some metric $D$, $D(a, b)\leq \varepsilon$.

\begin{theorem}[\cite{vershynin_2018} ]\label{thm:enet}
	Let $\mathcal{N}_{B}^{\varepsilon}$ be the smallest Euclidean $\varepsilon$-net of
    $\mathcal{B}:=\{x\in\mathbb{R}^q \ \vert\ \|x\|_2 \leq 1\}$ then $\left(\varepsilon^{-1}\right)^{q} \leq
    \vert \mathcal{N}_{\mathcal{B}}^{\varepsilon} \vert \leq \left(3\varepsilon^{-1}\right)^q.$
\end{theorem}

\begin{theorem}[\cite{vershynin_2018}]\label{thm:enet monotone}
	Let $\mathcal{K}\subset \mathcal{L}$. Let $\mathcal{N}^{\varepsilon}_\mathcal{K}$ be the smallest
    $\varepsilon$-net of $\mathcal{K}$ and $\mathcal{N}^{\epsilon/2}_\mathcal{L}$ be the smallest
    $\varepsilon/2$-net of $\mathcal{K}$. Then $ \vert \mathcal{N}^{\epsilon}_\mathcal{K}\vert \leq
    \vert \mathcal{N}^{\epsilon/2}_\mathcal{L}\vert$.
\end{theorem}

\section{Point-set registration via convex relaxation}\label{sec: behavior cvx}



We discuss the admissible corruption level to recover the ground truth rotation using a convex
relaxation of $\SO(d)$, i.e. solving \cref{prob:conv1} and \cref{prob:conv2}. Let the admissible
corruption threshold be defined as
\begin{align}
  \label{eq: definition p tilde2}
  \tilde{p}(d) := \left(1+\frac{B\left(d-1,\frac{1}{2}\right)}{B\left(\frac{d-1}{2},\frac{1}{2}\right)}\right)^{-1},
\end{align}
where $B(m,n)$ is the beta function. In \cref{sec:exact recovery},
we prove \cref{thm:sharpness convex}, showing that if $p<\tilde{p}(d) -o(1)$ the ground truth
rotation can be recovered for $N$ large enough. Similarly in \cref{sec:no recovery}, we prove
\cref{thm:imp. recovery}, showing that the ground truth is not recovered when $p>\tilde{p}(d) +
o(1)$.

The proof of \cref{thm:sharpness convex} and \cref{thm:imp. recovery} is based on constructing lower
and upper bounds for $L(A)-L(I)$, for any matrix $A$. We use the following inequality.
\begin{lemma}\label{lemma:bound difference norms}
	Let $u,v\in\mathbb{R}^{n}\setminus\{0\}$ then 
	\begin{align}
	\label{eq: lower bound difference norms}
	u^\top \frac{v}{\|v\|_{2}}\leq \|u +v\|_{2}-\|v\|_{2}\leq u^\top \frac{v}{\|v\|_{2}} +
    \frac{1}{2}\frac{\|u\|_{2}^3}{\sqrt{\|u\|_{2}^2\|v\|_{2}^2 - (u^\top v)^2}}
	\end{align} 
\end{lemma}
\begin{proof}
  See \cref{sec:bound difference norms}.
\end{proof}

\subsection{Proof of \cref{thm:sharpness convex}: Regime where exact recovery is possible}\label{sec:exact recovery}
We start this section by creating a lower bound of the expectation of $L(A)-L(I)$ in terms of the difference between $\tilde{p}(d)$ and $p$. 

\begin{lemma}[Success of \cref{prob:conv1} in expectation] \label{lemma:pop_suc_cvx}
  Let $A\in\mathbb{R}^{d\times d}\setminus I$ and $L(\cdot)$ the LUD as defined in \eqref{eq:LUDdef},
  then
  \begin{align}
    \E[ L(A) - L(I)] \geq \frac{\|A-I\|_{*}}{d}\frac{\left(\tilde{p}(d) - p\right)}{\tilde{p}(d)},
  \end{align}
  where $\|\cdot\|_{*}$ corresponds to the nuclear norm. 
\end{lemma}
\begin{proof}
  Using the definition of $L(\cdot)$ and \cref{lemma:bound difference norms}, we have 
  \begin{align}
	L(A)-L(I) &= \frac{1}{N}\sum_{i\in\mathcal{C}^{c}} \|(A-I)x_i\|_2 + \frac{1}{N}\sum_{i\in
      \mathcal{C}} \left(\|Ax_i - y_i\|_2- \|x_i - y_i\|_2\right)\nonumber\\ &\geq
    \frac{1}{N}\sum_{i\in\mathcal{C}^{c}} \|(A-I)x_i\|_2 + \frac{1}{N}\sum_{i\in\mathcal{C}}
    \left\langle (A-I)x_i, \frac{x_i-y_i}{\|x_i-y_i\|_{2}}\right\rangle\label{eq:deterministic lower
      bound convex}.
  \end{align}
  Given the distribution of $\{x_i\}_{i=1}^{N},\{y_i\}_{i=1}^{N}$ in \eqref{eq:prob_model0} and \eqref{eq:prob_model}, then 
  \begin{align}
    \E [L(A)-L(I)]\geq (1-p) \E \|(A-I)x\|_2 +\frac{p}{2} \left\langle
    A-I,\E\left[\frac{(x-y)(x-y)^\top}{\|x-y\|_{2}}\right]\right\rangle
  \end{align}
  for $x,y\in\Unif(\mathbb{S}^{d-1})$ independent. By \cref{lemma:exp_x_y sphere} we have
  \begin{align}
    \E\left[\frac{(x-y)(x-y)^\top}{\|x-y\|_{2}}\right]& =\frac{\E\|x-y\|_{2}}{d}I_{d} = \frac{2}{d}
    \frac{B\left(d-1,\frac{1}{2}\right)}{B\left(\frac{d-1}{2},\frac{1}{2}\right)}\ I_{d}.
  \end{align} 
  Similarly, \cref{lemma:bound exp. norm} provides the lower bound $\mathbb{E}\|(A-I)x\|_2\geq
  \|A-I\|_{*}/d$. Then,
  \begin{align}
    \E [L(A)-L(I)]\geq
    \frac{\|A-I\|_{*}}{d}\left[1-p\left(1+\frac{B\left(d-1,\frac{1}{2}\right)}{B\left(\frac{d-1}{2},\frac{1}{2}\right)}\right)\right].
  \end{align}
\end{proof}

\cref{lemma:pop_suc_cvx} implies that as $N\rightarrow \infty$, if $p<\tilde{p}(d)$,
\begin{align}
L(A)>L(I)\text{ for }A\neq I.
\end{align}
Therefore the unique minimum of $L(\cdot)$ is $I$. Hence one can exactly recover the ground truth
rotation via solving the convex problem if $p<\tilde{p}(d)$.

Now, we deduce a similar behavior when the sample size $N$ is finite. To prove \cref{thm:sharpness
  convex}, we show that the RHS of \eqref{eq:deterministic lower bound convex} is bounded away from
zero with high probability for most configurations of $\{x_i\}_{i=1}^N, \{y_i\}_{i=1}^N$. We use the
following result derived from Talagrand's inequality (\cref{theorem:talagrand}).

\begin{lemma}\label{lemma:concentration of residuals}
  Let $\{x_{i}\}_{i=1}^{n}$ be a set of independent random variables in $\mathbb{R}^{d}$ and
  $\{f_i\}_{i=1}^{n}$ a set of functions such that $f_{i}:\mathbb{R}^{d\times d}\times
  \mathbb{R}^{d}\rightarrow\mathbb{R}$ such that
  \begin{align}
	\label{eq: subgaussian increment}
	\|f_i(A,x_i)-f_i(B,x_i)\|_{\psi_{2}} \leq K \|A-B\|_{F}\quad
  \end{align}
  for all $\|A\|_{F}\leq 1,\|B\|_{F} \leq 1,\ i = 1,\dots,n$. Then, there exist a universal constant
  $c$ such that with high probability
  \begin{equation}
	\sup_{\|A\|_{F}\leq 1} \left\vert \sum_{i = 1}^n f_i(A,x_{i}) - \mathbb{E} f_i(A,x_{i})
    \right\vert \leq c\sqrt{n(d^2K^2+ \log n)}.
  \end{equation}
\end{lemma}
\begin{proof}
  See \cref{sec:con_res}.
\end{proof}

We can now give the proof of \cref{thm:sharpness convex}.
  \begin{proof}[Proof of \cref{thm:sharpness convex}]
  \cref{lemma:pop_suc_cvx} and \eqref{eq:deterministic lower bound convex} imply that
  \begin{align}
    L(A)-L(I) \geq& \frac{\|A-I\|_{*}}{d}\frac{\left(\tilde{p}(d) - p\right)}{\tilde{p}(d)}+
    \frac{1}{N}\sum_{i\in\mathcal{C}^{c}} \left(\|(A-I)x_i\|_2 - \E \|(A-I)x_i\|_2 \right)\nonumber\\& +
    \frac{1}{N}\sum_{i\in\mathcal{C}} \left(\left\langle (A-I)x_i,
    \frac{x_i-y_i}{\|x_i-y_i\|_{2}}\right\rangle - \E \left\langle (A-I)x_i,
    \frac{x_i-y_i}{\|x_i-y_i\|_{2}}\right\rangle \right).\label{eq:deterministic lower bound
      convex2}
  \end{align}
  
  To bound the last two terms in the RHS of \eqref{eq:deterministic lower bound convex2}, we first
  show that each term satisfies condition \eqref{eq: subgaussian increment} of
  \cref{lemma:concentration of residuals}. By triangle inequality, for $E,F\in\mathbb{R}^{d\times
    d}$
  \begin{align}
    \max_{x\in\mathbb{S}^{d-1}} \left \vert \|Ex\|_2 -\|Fx\|_2\right\vert&\leq \max_{x\in\mathbb{S}^{d-1}} \|(E-F)x\|_2 \leq \|E-F\|_{F}
  \end{align}
  and 
    \begin{align}
  \max_{x,u\in\mathbb{S}^{d-1}} \left \vert \left\langle Ex, u\right\rangle -\left\langle Fx, u\right\rangle\right\vert&\leq \max_{x\in\mathbb{S}^{d-1}} \|(E-F)x\|_2 \leq \|E-F\|_{F}.
  \end{align}
  Then $\|Ex\|_2 -\|Fx\|_2$ and $\left\langle Ex, u\right\rangle -\left\langle Fx, u\right\rangle$ are
  bounded random variables. Hence \eqref{eq: bounded subgaussian norm} implies for $ i\in\mathcal{C}^{c}$,
  \begin{align}
   \|\|Ex_{i}\|_2 -\|Fx_{i}\|_2\|_{\psi_2}\leq  (\log 2)^{-1/2} \|E-F\|_{F},
  \end{align}
  and for $i\in\mathcal{C}$,
    \begin{align}
 \left\|\left\langle Ex_i, \frac{x_i-y_i}{\|x_i-y_i\|_{2}}\right\rangle -\left\langle Fx_i, \frac{x_i-y_i}{\|x_i-y_i\|_{2}}\right\rangle\right\|_{\psi_2}\leq  (\log 2)^{-1/2} \|E-F\|_{F}.
  \end{align}
  
  Now, for any $E\in\mathbb{R}^{d\times d}$, we define
  \begin{align}
    \label{def: XE}
    X[E] :=\sum_{i\in\mathcal{C}^{c}} \left(\|Ex_i\|_2  - \E \|Ex_i\|_2 \right) + \sum_{i\in\mathcal{C}} \left(\left\langle E, \frac{x_i-y_i}{\|x_i-y_i\|_{2}}x_i^\top -  \E \left[\frac{x_i-y_i}{\|x_i-y_i\|_{2}}x_i^\top \right] \right\rangle \right).
  \end{align}
  Since $\|A-I\|_{*}>\|A-I\|_{F}$, then by inserting \eqref{def: XE} in the RHS of
  \eqref{eq:deterministic lower bound convex2} we get
  \begin{align}
    L(A)-L(I)&\geq \frac{\|A-I\|_{F}}{d}\left(\frac{\tilde{p}(d) - p}{\tilde{p}(d)}+
    \frac{d}{N}X\left[\frac{A-I}{\|A-I\|_{F}}\right]\right).
  \end{align}
  Using the result of \cref{lemma:concentration of residuals}, we have that
  \begin{align}
    \sup_{\|E\|_{F}=1} |X[E]| \leq  c\sqrt{N(d^2+\log N)}
  \end{align}
  with high probability. Therefore with high probability, for all $A\in\mathbb{R}^{d\times d} $
  \begin{align}\label{eq:lower_bound_LA}
    L(A)-L(I)&\geq \frac{\|A-I\|_{F}}{d\ \tilde{p}(d)}\left(\tilde{p}(d)- p -
    O\left(d\sqrt{\frac{(d^2+\log N)}{N}} \right)\right).
  \end{align}
  Then, with high probability, for all $A\neq I$,
  \begin{align}
    L(A)>L(I)\quad \text{if}\quad p + O\left(d\sqrt{\frac{(d^2+\log N)}{N}} \right)< \tilde{p}(d).
  \end{align}
  Hence the only minimizer of \cref{prob:conv1} in this regime is $I$.
\end{proof}

\subsection{Proof of \cref{thm:imp. recovery}: Regime where exact recovery is impossible}\label{sec:no recovery}
To prove \cref{thm:imp. recovery}, we find a matrix $C\in\mathrm{conv}\, \SO(d)$ and $s>0$ such that
when $ p>\tilde{p}(d) + o(1)$
\begin{eqnarray}\label{eq: matrix_C_impossible}
  L(C) < L(A),\quad \text{for all}\ A\ \text{s.t.}\  \|A-I\|_{F}<s.
\end{eqnarray}
with high probability. 
This implies that, if the corruption level is large enough, exact recovery of the ground truth rotation is impossible when minimizing $L(\cdot)$ over
$\mathrm{conv}\, \SO(d)$.  First, we find such matrix $C$ that satisfies \eqref{eq: matrix_C_impossible} in expectation.
\begin{lemma}[Failure of \cref{prob:conv2} in expectation, Part 1.]\label{lemma:pop_fail_cvx}
  Let $L(\cdot)$ the LUD as defined in \eqref{eq:LUDdef}, $p>\tilde{p}(d)$ and 
  \begin{align}\label{eq:def_l_str}
	\lambda^{*}:= \left(\frac{p-\tilde{p}(d)}{\tilde{p}(d)\cdot p} \right)
    \frac{B\left(\frac{d-1}{2},\frac{1}{2}\right)}{B\left(\frac{d-2}{2},\frac{1}{2}\right)}\leq 1
  \end{align}
  then 
  \begin{align}
	\E[L((1-\lambda^*)I)-L(I)]\leq -\frac{\lambda^*}{2\tilde{p}}\left(p-\tilde{p}(d)\right)<0.
  \end{align}
\end{lemma}

\begin{proof} 
  We want to show $(1-\lambda)I$ attains better cost than $I$ for some $\lambda>0$. To this end, we
  find an upper bound for $\E[L((1-\lambda)I)-L(I)]$ for any $\lambda>0$ and then we get
  $\lambda^{*}$ that minimizes this upper bound. By the definition of $L(\cdot)$, we get
  \begin{align}\label{eq:L_lst}
    L((1-\lambda)I) - L(I)&= (1-p)\lambda +\frac{1}{N}\sum_{i\in\mathcal{C}} \|-\lambda x_i +
    x_i-y_i\|_{2}-\|x_i-y_i\|_{2}.
  \end{align}
  \cref{lemma:bound difference norms} provides the upper bound for $i\in\mathcal{C}$
  \begin{align}
    \|-\lambda x_i + x_i-y_i\|_{2}-\|x_i-y_i\|_{2}\leq -\frac{\lambda}{2} \|x_{i}-y_{i}\|_{2} +
    \lambda^2\frac{1}{\|x_i-y_i\|_{2}\|x_i+y_i\|_{2}}.
  \end{align}
  
  Although $\|x_i-y_i\|_{2}$ or $\|x_i+y_i\|_{2}$ may be zero, if $d\geq 3$, \cref{lemma:exp_x_y
    sphere} states
  \begin{align}
	\E \|x_i-y_i\|_{2}
    \ =\ 2\frac{B\left(d-1,\frac{1}{2}\right)}{B\left(\frac{d-1}{2},\frac{1}{2}\right)},\quad
    \E\left[
      \frac{1}{\|x_i-y_i\|_{2}\|x+y\|_{2}}\right]\ =\ \frac{B\left(\frac{d-2}{2},\frac{1}{2}\right)}{2B\left(\frac{d-1}{2},\frac{1}{2}\right)}.
  \end{align}
  Therefore for all $\lambda>0$,
  \begin{align}
	\label{eq: bound for all lambda}
	\E[ L((1-\lambda)I) - L(I)]&\leq -\lambda\frac{(p-\tilde{p}(d))}{\tilde{p}(d)} +
    \frac{p\lambda^2}{2}\frac{B\left(\frac{d-2}{2},\frac{1}{2}\right)}{B\left(\frac{d-1}{2},\frac{1}{2}\right)}.
	\end{align}
  In particular, taking $\lambda =\lambda^*$ as defined in \eqref{eq:def_l_str} the RHS of \eqref{eq: bound for
    all lambda} is minimized. Moreover, since the Beta function satisfies
   \begin{align}
    B(q,1/2)\leq B(r,1/2) \text{ if } q\geq r\geq 1/2,
   \end{align}
then $\lambda^*<1$. Given that $0,I\in\mathrm{conv}\, \SO(d)$, then
  $(1-\lambda^*)I\in\mathrm{conv}\, \SO(d)$.
\end{proof}

\begin{corollary}[Failure of \cref{prob:conv2} in expectation, Part 2]\label{cor:pop_fail_cvx}
  Let $L(\cdot)$ the LUD as defined in \eqref{eq:LUDdef}, $\lambda^*$ as defined in
  \eqref{eq:def_l_str} and $p>\tilde{p}(d)$. For all $A\in\mathbb{R}^{d\times d}$ such that
  $\|A-I\|_{F} \leq \sqrt{d}\lambda^{*}/4$ then
  \begin{align}
	\E[ L(A) - L((1-\lambda^*)I)] \geq \frac{p-\tilde{p}(d)}{\tilde{p}(d)}\frac{\lambda^*}{4}>0.
  \end{align}
\end{corollary}

\begin{proof}
  Follows directly from \cref{lemma:pop_suc_cvx} and \cref{lemma:pop_fail_cvx}.
\end{proof}

\cref{cor:pop_fail_cvx} implies that when $N\rightarrow\infty$, any $A^*$ minimizer of
\cref{prob:conv2} satisfies $\|A^{*}-I\|_{F}> \lambda^*\sqrt{d}/4 = O(p-\tilde{p}(d))$.

We now prove \cref{thm:imp. recovery}.
  \begin{proof}[Proof of \cref{thm:imp. recovery}]
  \cref{lemma:pop_fail_cvx} and \eqref{eq:L_lst} imply that
  \begin{align}\label{eq:bound_L_st}
	L(I) - L((1-\lambda^{*})I)\geq \frac{\lambda^*}{2\tilde{p}}\left(p-\tilde{p}\right)
    -\frac{1}{N}\sum_{i\in\mathcal{C}} (g(x_{i},y_{i},\lambda^{*}) -\E g(x_{i},y_{i},\lambda^{*}) )
  \end{align}
  where $g(x_{i},y_{i},\lambda^{*}) := \|(1-\lambda^{*}) x_i-y_i\|_{2}-\|x_i-y_i\|_{2} $. By triangle inequality,
  \begin{align}
	\max_{x,y\in\mathcal{S}^{d-1}}\left\vert \|(1-\lambda^{*})x-y\|_2 - \|x-y\|_2\right\vert \leq \max_{x\in\mathcal{S}^{d-1}}\|\lambda^{*}x\|_2=\lambda^*.
  \end{align}
  Therefore \eqref{eq: bounded subgaussian norm} implies that $\|g(x_i,y_i,\lambda^{*})\|_{\psi_2}
  \leq (\log2)^{-1/2} \lambda^*$ for all $i\in\mathcal{C}$. Using Hoeffding's inequality
  (\cref{theorem:hoeffding}) we get
  \begin{equation}
	\label{eq:res2 estimate imposs}
	\left \vert\sum_{i\in\mathcal{C}}\left[g(x_i,y_i,\lambda^*) -
      \mathbb{E}g(x_i,y_i,\lambda^*)\right]\right \vert \leq c\lambda^{*}\sqrt{N\log N}
  \end{equation}
  with high probability. \eqref{eq:res2 estimate imposs} implies that the LHS of \eqref{eq:bound_L_st}
  is bounded by
  \begin{align}
	L(I) - L((1-\lambda^{*})I)\geq \frac{\lambda^*}{2\tilde{p}(d)}\left(p-\tilde{p}(d)
    -O\left(\sqrt{\frac{\log N}{N}}\right)\right).
  \end{align}
  
  From \eqref{eq:lower_bound_LA} in the proof of \cref{thm:sharpness convex}, we have with high
  probability, for all $A\in\mathbb{R}^{d\times d}$
  \begin{align}
	L(A)-L(I)&\geq -\frac{\|A-I\|_{F}}{\sqrt{d}\ \tilde{p}(d)}\left(p-\tilde{p}(d) +
    O\left(\sqrt{d\frac{(d^2+\log N)}{N}} \right)\right),
  \end{align}
  since $\|A-I\|_{*}\leq\sqrt{d}\|A-I\|_{F}$. Hence, for all $A\in\mathbb{R}^{d\times d}$ such that
  $\|A-I\|_{F} \leq \sqrt{d}\frac{\lambda^{*}}{4}$,
  \begin{align}
    L(A) - L((1-\lambda^*)I) \geq\frac{\lambda^{*}}{4} \left(p-\tilde{p}(d) +
    O\left(\sqrt{d\frac{(d^2+\log N)}{N}} \right)\right),
  \end{align} 
  with high probability. Therefore, if $p  >\tilde{p}(d)+O\left(\sqrt{d(d^2+\log N){N}^{-1}}\right)$, then
  \begin{align}
    L(A)>L((1-\lambda^*)I)
  \end{align}
  with high probability. Hence if $A^*$ is a minimizer of \cref{prob:conv2}, then $\|A^{*}-I\|_{F}>
  \lambda^*\sqrt{d}/4 = O(p-\tilde{p}(d))$.
\end{proof}

\section{Finite time dynamical systems in $SO(d)$}\label{sec:fin time SOd}
To prove \cref{thm:converg. nonconvex}, we want to show the dynamical system \eqref{eq:dyn_sysLUD}
converges in finite time, i.e. there exists $T<\infty$ such that $R(t) = I$ for all $t>T$. In this
section, we characterize the sufficient conditions for finite time convergence of any dynamical
system in $\SO(d)$ of the form
\begin{align}
\label{eq: DI in SO}
\frac{dR}{dt}\in R\cdot \mathcal{S}(R),\quad \mathcal{S}(R)\subseteq \mathcal{S}_{\mathrm{skew}}(d)
\end{align}
in terms of the dynamics of $\|\log R(t)\|_{2}$. We start with an useful characterization of
$R\in\SO(d)$ in terms of the principal angles of rotation. We denote the skew-symmetric matrix and
the identity matrix in $\mathbb{R}^{2\times 2}$
\begin{align}
  A_{2} := \left[\begin{array}{cc}
      0 & -1\\
      1& 0
    \end{array}\right],\quad I_{2} := \left[\begin{array}{cc}
      1 & 0\\
      0& 1
    \end{array}\right],
\end{align}
respectively.
\begin{lemma}[Planar decomposition in $\SO(d)$]\label{lemma: planar decomposition R}
  Any $R\in\SO(d)$ can be written as  
  \begin{align}
	\label{eq: planar decomposition R}
	R - I = \sum_{i=1}^{\left\lfloor d/2 \right\rfloor}U_{i}(R_{\sigma_{i}} - I_{2})
    U_{i}^\top,\quad R_{\theta} := \left[\begin{array}{cc} \cos\theta& -\sin\theta\\ \sin\theta&
        \cos\theta
	  \end{array}\right]
  \end{align}  
  with $\pi\geq\sigma_{1}\geq\dots\geq\sigma_{\left\lfloor d/2 \right\rfloor}\geq 0 $ and
  $U=\left[\begin{array}{cccc} U_1 & U_2 & \dots & U_{\left\lfloor d/2 \right\rfloor}
	\end{array}\right]$ such that $U^\top U = I$  and $U_i\in\mathbb{R}^{d\times 2}$.
\end{lemma}

\begin{proof}
  For any rotation $R$, we can define the principal logarithm such that $\log R\in\pi\mathcal{B}_{\mathrm{skew}}(d)$, as defined in \eqref{eq:def_B_skew}. Therefore, the SVD decomposition of $\log R$ satisfies
  \begin{align}
	\log R = \sum_{i=1}^{\left\lfloor d/2 \right\rfloor} \sigma_{i} U_{i} A_{2} U_{i}^\top
  \end{align}
  for $\pi\geq\sigma_{1}\geq\dots\geq\sigma_{\left\lfloor d/2 \right\rfloor}\geq 0 $ and
  \[
  U=\left[\begin{array}{cccc} U_1 & U_2 & \dots & U_{\left\lfloor d/2 \right\rfloor}
	\end{array}\right]
  \]
  such that $U^\top U = I$. Notice that for any $\theta\in[0,\pi]$,
  \begin{align}
	\exp(\theta \cdot A_2) = \cos(\theta) I_{2} + \sin(\theta)A_{2} = R_{\theta}.
  \end{align}
  Then, using $R=\exp(\log R)$, we get decomposition \eqref{eq: planar decomposition R} for $R$.
\end{proof}

\cref{lemma: planar decomposition R} provides a decomposition of any rotation of dimension $d$ as a
sum of planar rotations in orthogonal planes. Notice that if $d$ is odd, \eqref{eq: planar
  decomposition R} implies that $R$ has at least one invariant direction. In addition,
$\sigma_{1} = \|\log R\|_{2}$ characterizes the largest angle of rotation of $R$. Furthermore, $\|\log R\|_{2}$ provides the following bounds:
\begin{align}
  \|\log R\|_{2}\leq D_{\SO}(R,I)\leq \sqrt{d}\|\log R\|_{2}.
\end{align}
where $D_{\SO}(R,I)$ denotes the Riemannian distance between $R$ and $I$ as defined in \eqref{eq:riemannian_distance}.

Now we characterize the rate of change of $\|\log R(t)\|_2$ assuming that $R(t)$ follows the dynamics 
\eqref{eq: DI in SO}.

\begin{lemma}[Dynamics of $\|\log R\|_{2}$]\label{lemma:dyn_SO}
  Let $R(t)\in SO(d)$ be a solution of the differential inclusion \eqref{eq: DI in SO}, then $\|\log
  R(t)\|_{2}$ is absolutely continuous and satisfies
  \begin{align}
	\frac{d\|\log R(t)\|_{2}}{dt} \in \{a\mid \exists S \in\mathcal{S}(R(t)), \langle
    S,UA_{2}U^\top \rangle = 2a\ \forall U\in \mathcal{U}(R(t))\}\
  \end{align}
  almost everywhere, where
  \begin{align}
	\label{eq:def UR}
	\mathcal{U}(R):=\left\{U\in\mathbb{R}^{d\times 2}\mid U^\top U = I_2,\quad U^\top (\log R)U = \|\log R\|_{2}A_{2}\right\}.
  \end{align}
\end{lemma}
\begin{proof} See \cref{sec: app_proof_dyn}.
\end{proof}

Ultimately, we want to show $\|\log R(t)\|_{2}$ converges to zero as $t$ increases. A sufficient
condition for this is summarized in the the following lemma.
\begin{lemma}[Finite-time convergence in $\SO(d)$]\label{lemma:conv dyn sys SO}
  Let $R(t)\in SO(d)$ be a solution of the differential inclusion \eqref{eq: DI in SO}
  with initial condition $R(0)$. Suppose there exist $g\in\mathcal{C}^{1}[0,\pi]$ such that 
  \begin{align}\label{eq:conv_dyn_SO_1}
    \max_{S\in\mathcal{S}(R(t)),U\in\mathcal{U}(R(t))} \langle S,UA_{2}U^\top \rangle  \leq -2g(\|\log R(t)\|_{2})
  \end{align} 
  holds for $R(t)\neq I$ almost everywhere, $g(s)>0$ for $0\leq s\leq \|\log R(0)\|_{2}$ and
  $\mathcal{U}(R(t))$ as defined in \eqref{eq:def UR}. Then for all $t\geq T(\|\log R(0)\|_{2})$,
  $R(t)=I$, where
  \begin{align}\label{eq:conv_dyn_SO_4}
	T(s) := \int_{0}^{s}\frac{1}{g(u)}du < \infty.
  \end{align}
\end{lemma}

\begin{proof}
  \cref{lemma:dyn_SO} implies that $\|\log R(t)\|_{2}$ is absolutely continuous, i.e.
  \begin{align}\label{eq:abs_cont}
    \|\log R(t)\|_{2} &=\|\log R(0)\|_{2}+\int_{0}^{t}\frac{d}{d\tau} \|\log R(\tau)\|_{2} d\tau
  \end{align}
  and if $\frac{d\|\log R(t)\|_{2}}{dt}$ exists when $R(t) = I$ then 
  \begin{align}
    \frac{d\|\log R(t)\|_{2}}{dt} \in \{a\mid \exists S \in\mathcal{S}(t), \langle S,UA_{2}U^\top
    \rangle = 2a\ \forall U\in \mathcal{U}(I)\}\ \subseteq\{0\}.
  \end{align}
  Therefore, condition \eqref{eq:conv_dyn_SO_1} and $g(s)>0$ for $0\leq s\leq \|\log R(0)\|_{2}$ imply
  \begin{align}
    \|\log R(t_2)\|_{2} \leq \|\log R(t_1)\|_{2}\text{ for }t_{1}\leq t_{2},\label{eq:conv_dyn_SO_3}
  \end{align} 
  and we have strict inequality when $R(t_1)\neq I$. In particular, if there exists $T$ such that
  $R(T)=I$, then $R(t)=I$ for all $t\geq T$.
  
  When the initial condition is not $I$, then \eqref{eq:conv_dyn_SO_3} implies that for $\|\log
  R(0)\|_{2}>0$, we can upper bound the minimum time $\tau$ such that $\|\log R(\tau)\|_{2} =0$ by
  \begin{align}
    \tau = \int_{0}^{\tau} ds \leq \int_{0}^{\|\log R(0)\|_{2}}\frac{1}{g(u)}du =: T(\|\log R(0)\|_{2}).
  \end{align}
  Then, $R(t)=I$ for $t\geq T(\|\log R(0)\|_{2})$ as defined in  \eqref{eq:conv_dyn_SO_4}.
\end{proof}

In the rest of the paper, we will show finite time convergence of $R(t)$ to $I$ by checking the
condition in \cref{lemma:conv dyn sys SO}.


\section{Proof of \cref{thm:converg. nonconvex}: Exact recovery via non-convex optimization}\label{sec:non-convex recovery}
We examine exact recovery via non-convex optimization for solving \cref{prob:LUD}, i.e. by explicitly
restricting the transformation between the two point clouds to lie on $\SO(d)$. As mentioned previously,
we take a dynamical system view-point of generalized gradient descent on the manifold $\SO(d)$ by
considering the Riemannian gradient flow
\begin{equation}
\label{eq:dyn. system}
\frac{d R}{dt}(t) \in - R(t)\cdot\mathrm{skew}(R(t)^\top \bar{\partial} L(R(t)))\subset
T_{R(t)}\SO(d)
\end{equation}
where $\bar{\partial} L(R) $ is the Euclidean generalized gradient of $L(\cdot)$ given by 
\begin{align}\label{eq:subgradient2}
  \bar{\partial} L(R) := \frac{1}{N} \sum_{i\in\mathcal{C}^{c}} \alpha_i(R) x_i^T + \frac{1}{N}
  \sum_{i\in\mathcal{C}} \beta_i(R) x_i^T,
\end{align}
and $\alpha_i$'s and $\beta_i$'s are defined as
\begin{align} 
  \label{def: alphai}
  \text{for }i\in\mathcal{C}^{c},\quad \alpha_i(R) &:=
  \begin{cases}
    \left\{\frac{Rx_i-x_i}{\|Rx_i-x_i\|_2}\right\}, & \text{if}\ Rx_i-x_i\neq 0 \\
    \left\{\theta_i\in\mathbb{R}^d\mid \|\theta_i\|_2\leq 1\right\}, & \text{otherwise};
  \end{cases}\\
  \label{def: betai}
  \text{for }i\in\mathcal{C},\quad \beta_i(R) &:=
  \begin{cases}
    \left\{\frac{Rx_i-y_i}{\|Rx_i-y_i\|_2}\right\}, & \text{if}\ Rx_i-y_i\neq 0 \\
    \left\{\theta_i\in\mathbb{R}^d\mid \|\theta_i\|_2\leq 1\right\}, & \text{otherwise}.
  \end{cases}.
\end{align} 
Although $\alpha_{i}(R),\beta_{i}(R)$ are set-valued functions, we can also consider
$\alpha_{i}(R),\beta_{i}(R)$ as single value functions, by associating each $\alpha_i,\ \beta_i$
with an arbitrary element from $\left\{\theta_i\in\mathbb{R}^d\mid \|\theta_i\|_2\leq 1\right\}$
whenever its value is not uniquely defined.
 
\cref{lemma:conv dyn sys SO} states that the differential inclusion \eqref{eq:dyn. system} with
initial condition $R(0)$ converges to $I$ in finite time if there exist
$g:[0,\pi]\rightarrow\mathbb{R}$ continuously differentiable such that
\begin{align}\label{eq:suff_cond_conv}
  \min_{U\in\mathcal{U}(R)}\left\langle \bar{\partial}{L}(R),R UA_{2}U^\top \right\rangle\geq 2 g(\|\log R\|_{2})
\end{align}
holds for all $R\in\SO(d)\setminus I $ and $g(s)>0$ for $s\in(0,\|\log R(0)\|_{2})$, where
\begin{align}
  \label{eq:def UR2}
  \mathcal{U}(R):=\left\{U\in\mathbb{R}^{d\times 2}\mid U^\top U = I,\quad U^\top (\log R)U = \|\log
  R\|_{2}A_{2}\right\}.
\end{align}
In this section we show \eqref{eq:suff_cond_conv} holds for a particular $g$ in \cref{lemma:sample
  gradient nonconvex}.
\begin{lemma}
  \label{lemma:sample gradient nonconvex} 
  With high probability, for all $R\in \SO(d)\setminus I$ and $U\in\mathcal{U}(R)$
  \begin{equation}
	\label{eq:sample gradient nonconvex}
	\left\langle \bar{\partial}{L}(R), R\ UA_2U^\top\right\rangle
    \geq\frac{2(1-p)}{d}\cos\left(\frac{\|\log R\|_{2}}{2}\right) - O\left( \sqrt{\frac{d(\log N +
        d)}{N}}\right)
  \end{equation}
\end{lemma}
Before proving \cref{lemma:sample gradient nonconvex}, we first show that \cref{thm:converg. nonconvex}
directly follows from \cref{lemma:conv dyn sys SO} and \cref{lemma:sample gradient
  nonconvex}.

  \begin{proof}[Proof of \cref{thm:converg. nonconvex}]
  \cref{lemma:sample gradient nonconvex} implies that there exists $\gamma >0$ such that
  \begin{align}
  \gamma= O\left(\frac{d}{1-p}\sqrt{\frac{d(\log N + d)}{N}}\right)
  \end{align}
  and 
   for all
  $R\in \SO(d)\setminus I$ and $U\in\mathcal{U}(R)$
  \begin{align}
    \cos\left(\frac{\|\log R\|_{2}}{2}\right) - \frac{d}{2(1-p)} \left\langle \bar{\partial}{L}(R),
    R\ UA_2U^\top\right\rangle\leq \gamma
  \end{align}
  with high probability. Then, $g(s)=((1-p)/d)(\cos(s/2) -\gamma)$ satisfies
  \eqref{eq:suff_cond_conv} with high probability. If $\|\log R(0)\|_{2}<\pi(1 -\sqrt{\gamma})$, then
  $g(\|\log R(0)\|_{2})>0$ and \cref{lemma:conv dyn sys SO} implies that $R(t)=I$ for all $t\geq
  T(\|\log R(0)\|_{2})$, where
  \begin{align}
    \label{eq:Ts}
    T(s)&=\frac{d}{(1-p)}\int_{0}^{s}\frac{1}{\cos\left(\frac{\tau}{2}\right) - \gamma} d\tau\nonumber\\& =
    \frac{2d}{(1-p)}\left(\cosh^{-1}\left(\sec\left(\frac{s}{2}\right)\right)+O\left(\sqrt{\gamma}\right)\right)
  \end{align}
  for $s\in[0, \|\log R(0)\|_{2}]$.
\end{proof}

In the remaining of this section, we prove \cref{lemma:sample gradient nonconvex}. First, in \cref{sec:pop. version nonconv}, we prove
\cref{lemma:bound gradient nonconvex}, that shows that inequality \eqref{eq:sample gradient nonconvex} holds in expectation. This leads to the proof of
\cref{lemma:sample gradient nonconvex} in \cref{sec: sample version nonconv}; the proof follows from
\cref{lemma:behav. good gradient non convex} (\cref{sec:behav. good gradient non convex}) and
\cref{lemma:behav. bad gradient non convex} (\cref{sec:behav. bad gradient non convex}) where we
estimate the deviation of the LHS of \eqref{eq:sample gradient nonconvex} from its expectation.

\subsection{Expectation version of \cref{lemma:sample gradient nonconvex}}\label{sec:pop. version nonconv}
In this section, we show that inequality \eqref{eq:sample gradient nonconvex} holds in expectation
\begin{lemma}[Expectation version of \cref{lemma:sample gradient nonconvex}]\label{lemma:bound gradient nonconvex}
	For all $R \in \SO(d)\setminus I$ and $U\in\mathcal{U}(R)$
	\begin{align}
      \label{eq:bound gradient nonconvex}
	  \frac{1}{N}\E\left\langle\sum_{i \in\mathcal{C}^c}\alpha_{i}(R)x_{i}^{\top} +\sum_{i
        \in\mathcal{C}}\beta_{i}(R)x_{i}^{\top} , R\ UA_2U^\top\right\rangle &\geq \frac{2(1-p)}{d}
      \cos\left(\frac{\|\log R\|_{2}}{2}\right).
	\end{align}
\end{lemma}

To prove \cref{lemma:bound gradient nonconvex},  we first introduce the following lemma.
\begin{lemma}\label{lemma:good gradient nonconvex}
	Let $R \in \SO(d)\setminus I$, then for all $U\in \mathcal{U}(R)$ and $i\in\mathcal{C}^{c}$
	\begin{align}
	\label{eq:expected subgradient}
	\left \langle \alpha_{i}(R), R\cdot UA_{2}U^\top x_i\right \rangle &\geq \|UA_{2}U^\top
    x_i\|^2_{2}\cos\left(\frac{\|\log R\|_{2}}{2}\right)
	\end{align}
\end{lemma}
\begin{proof}
  First, we consider the case when $Rx_i=x_i$. Then $x_i$ is an invariant direction of $R$. By
  \cref{lemma: planar decomposition R}, $(\log R)x_i=0$. Since, $R\neq I$, then
  $\mathrm{ker}(UA_{2}U^\top )\supseteq\mathrm{ker}(\log R)$ and $UA_{2}U^\top x_i =0$. Therefore
  \eqref{eq:expected subgradient} trivially holds.
  
  Now, we consider the case when $Rx_i\neq x_i$. Then $ \alpha_i(R)=(Rx_i - x_i)/\|Rx_i -
  x_i\|_2$.  Since $UA_{2}U^\top\in\mathcal{S}_{\mathrm{skew}}$, then $\langle x_i, UA_{2}U^\top
  x_i\rangle = 0$. Therefore,
  \begin{align}
	\left\langle \frac{Rx_i - x_i}{\|Rx_i - x_i\|_2} , R\cdot UA_{2}U^\top x_i \right\rangle &=
    \frac{-x_i^T R\cdot UA_{2}U^\top x_{i} }{\sqrt{2 - 2x_i^T R x_i}}. \label{eq:subgradient good}
  \end{align}
  
  By definition of $\mathcal{U}(R)$, $R\cdot U = U \exp(\|\log R\|_{2}A_{2})$. Let
  $\tilde{x}_{i}=U^\top x_{i}$. Then,
  \begin{align}
	\label{eq:good numerator}
	-x_i^T R\cdot UA_{2}U^\top x_{i} &= -\tilde{x}_i^T \exp\left(\|\log R\|_{2} A_2 \right) A_2\tilde{x}_{i}\nonumber\\
	&= \sin(\|\log R\|_{2})\ \tilde{x}_i^T \tilde{x}_{i}
  \end{align}
  and $\tilde{x}_i^T \tilde{x}_{i} = \|UA_{2}U^\top x_{i}\|_{2}^{2}$. Similarly, using the planar
  decomposition of $R$ in \cref{lemma: planar decomposition R},
  \begin{align}
	\label{eq:good denominator}
	1-x_{i}^\top R x_{i} &= \sum_{j=1}^{\lfloor d/2\rfloor}x_{i}^\top U_{j}(I_{2}-\exp\left(\sigma_{j} A_{2}\right))U_{j}^\top x_{i}\nonumber\\
	&= \sum_{j=1}^{\lfloor d/2\rfloor}(1-\cos(\sigma_j))x_{i}^\top U_{j}U_{j}^\top x_{i}\nonumber\\
	&\leq  1-\cos(\|\log R\|_{2}),
  \end{align}
  since $\sigma_{j}\leq \|\log R\|_{2}\leq \pi$ for all $j= 1,\dots,\lfloor d/2\rfloor$.
  Inserting \eqref{eq:good numerator} and \eqref{eq:good denominator} in RHS of \eqref{eq:subgradient
    good}, we get
  \begin{align}
	\left\langle \frac{Rx_i - x_i}{\|Rx_i - x_i\|_2} x_i^T , R\cdot UA_{2}U^\top \right\rangle &\geq
    \frac{\sin(\|\log R\|_{2})\|UA_{2}U^\top x_{i}\|_{2}^{2}}{2\sin(\|\log R\|_{2}/2)}.
  \end{align}
\end{proof}

Using \cref{lemma:good gradient nonconvex}, we can now prove \cref{lemma:bound gradient nonconvex}.
  \begin{proof}[Proof of \cref{lemma:bound gradient nonconvex}]
  Let
  \begin{align}\label{eq:def_M}
    M = \frac{1}{N}\sum_{i \in\mathcal{C}^c}\alpha_{i}(R)x_{i}^{\top} +\frac{1}{N}\sum_{i
      \in\mathcal{C}}\beta_{i}(R)x_{i}^{\top}
  \end{align}
  then \cref{lemma:good gradient nonconvex} implies
  \begin{align}\label{eq:lower_bound_inner}
    \left\langle M, R\cdot UA_{2}U^\top \right \rangle \geq& \cos\left(\frac{\|\log
      R\|_{2}}{2}\right)\frac{1}{N}\sum_{i\in\mathcal{C}^{c}} \|UA_{2}U^\top x_{i}\|_{2}^{2}\nonumber\\ &+
    \frac{1}{N}\sum_{i\in\mathcal{C}} \left\langle \beta_{i}(R), R\cdot UA_{2}U^\top x_i
    \right\rangle.
  \end{align}
  Given a fixed rotation $R$, $\beta_{i}(R) = (Rx_i-y_i)/\|Rx_i-y_i\|_{2}$ almost
  everywhere. Additionally, for $i\in\mathcal{C}$, rotation invariance implies that $Rx_{i}$ and
  $y_{i}$ are i.i.d random variables. Therefore, for $i\in\mathcal{C}$,
  \begin{eqnarray}
    \E\left[\beta_{i}(R)\ (Rx_{i})^\top \right] =\frac{1}{2}
    \E\left[\frac{(Rx_i-y_i)(Rx_i-y_i)^\top}{\|Rx_i-y_i\|_{2}}\right] =
    \frac{\E[\|x_{i}-y_{i}\|_{2}]}{2d}I
  \end{eqnarray}
  and $\E\left\langle \beta_{i}(R) , R\cdot UA_{2}U^\top x_i \right\rangle = 0 $.
  
  On the other hand, for all $i\in\mathcal{C}^{c}$, $\E \|UA_{2}U^\top x_{i}\|_{2}^{2} =
  \|UA_{2}U^\top\|_{F}^2/d=2/d$. Hence
  \begin{align}
    \E\left\langle M, R\cdot UA_2U^\top\right\rangle &\geq\frac{2(1-p)}{d} \cos\left(\frac{\|\log
      R\|_{2}}{2}\right).
  \end{align}
\end{proof}

\cref{lemma:bound gradient nonconvex} and \cref{lemma:conv dyn sys SO} imply that as $N\rightarrow
\infty$, for all $p<1$ and $\|\log R\|_{2}<\pi$, the dynamical system \eqref{eq:dyn. system}
converges to $I$ in finite time.

\subsection{Proof of \cref{lemma:sample gradient nonconvex} via concentration inequalities}\label{sec: sample version nonconv}

\cref{lemma:good gradient nonconvex} provides the behavior in expectation of LHS of \eqref{eq:sample
  gradient nonconvex}.  To prove \cref{lemma:sample gradient nonconvex}, we bound the deviation of
\eqref{eq:sample gradient nonconvex} away from the population limit. We deal separately with the terms
concerning corrupted and uncorrupted points in \cref{lemma:behav. good gradient non convex} and
\cref{lemma:behav. bad gradient non convex} respectively.

\begin{lemma}[Deviation from expectation for uncorrupted points]\label{lemma:behav. good gradient non convex}
  Let $\mathcal{C}^{c}$ be the index set of uncorrupted points, then
  \begin{align}
	\label{eq:behav. good gradient nonconvex}
    \sup_{\|A\|_{F}=1}\left|\frac{1}{N}\sum_{i\in \mathcal{C}^{c}}\left[\|A x_i\|^2_{2} - \mathbb{E}\|A x_i\|^2_{2}\right]\right|\leq c\sqrt{\frac{d+\log N}{N}}	
  \end{align}
  with high probability and $c$ is a universal constant. 
\end{lemma}

\begin{lemma}[Deviation from expectation for corrupted points]\label{lemma:behav. bad gradient non convex}
  Let $\mathcal{C}$ be the index set of corrupted points, then with high probability, for all
  choices of $\beta_{i}(R)$ as defined in \eqref{def: betai},
  \begin{align}
	\label{eq:behav. bad gradient nonconvex}
	\sup_{R\in\SO(d),S\in\mathcal{B}_{\mathrm{skew}}}
    \left|\frac{1}{N}\sum_{i\in\mathcal{C}}\left\langle RSx_i, \beta_{i}(R)\right\rangle
    \right|&\leq c \sqrt{\frac{d\log (N)}{N}}
  \end{align}
  with universal constant $c$.
\end{lemma}

\cref{lemma:sample gradient nonconvex} is a direct consequence of \cref{lemma:behav. good gradient
  non convex} and \cref{lemma:behav. bad gradient non convex}.

  \begin{proof}[Proof of \cref{lemma:sample gradient nonconvex}]
  By \eqref{eq:lower_bound_inner} and \cref{lemma:bound gradient nonconvex}, for $M$ defined as
  in \eqref{eq:def_M} and any value choice of $\alpha_i(R),\beta_{i}(R)$,
  \begin{align}
	\left\langle M, R\cdot UA_{2}U^\top \right \rangle &\geq\cos\left(\frac{\|\log
      R\|_{2}}{2}\right)\left( \frac{2(1-p)}{d}\right)\nonumber\\
      &\ +\cos\left(\frac{\|\log
      	R\|_{2}}{2}\right)\left( \frac{1}{N}\sum_{i\in
      	\mathcal{C}^{c}}\left[\|UA_{2}U^\top x_i\|^2_{2} - \mathbb{E}\|UA_{2}U^\top
      x_i\|^2_{2}\right]\right)\nonumber\\ 
       &\ +\frac{1}{N}\sum_{i\in\mathcal{C}}\left\langle R\cdot
    UA_{2}U^\top x_i, \beta_{i}(R)\right\rangle.
  \end{align}
  Since $UA_{2}U^\top\in\mathcal{S}_{\mathrm{skew}}$ such that $\|UA_{2}U^\top\|_{2} = 1$ and
  $\|UA_{2}U^\top\|^2_{F} = 2$, \cref{lemma:behav. good gradient non convex} and
  \cref{lemma:behav. bad gradient non convex} imply that with high probability, for all $U\in
  \mathcal{U}(R)$, and all value choices of $\alpha_i(R),\beta_{i}(R)$,
  \begin{equation}
	\left\langle M, R\cdot UA_2U^\top\right\rangle \geq\frac{2(1-p)}{d}\cos\left(\frac{\|\log
      R\|_{2}}{2}\right) - O\left( \sqrt{\frac{d(\log N + d)}{N}}\right).
  \end{equation}
\end{proof}

\subsubsection{Proof of \cref{lemma:behav. good gradient non convex}: Deviation from expectation concerning uncorrupted points}\label{sec:behav. good gradient non convex}

In this section, the proof of \cref{lemma:behav. good gradient non convex} is detailed.\newline

  \begin{proof}[Proof of \cref{lemma:behav. good gradient non convex}]
  We first show that $\|A x_i\|^2_{2} - \mathbb{E}\|A x_i\|^2_{2}$ satisfies condition \eqref{eq:
    subgaussian increment} of \cref{lemma:concentration of residuals}. Let
  $S,T\in\mathbb{R}^{n\times n}$ such that $\|S\|_{F}=\|T\|_{F} =1$, and $x,y\in \mathbb{S}^{d-1}$,
  then
  \begin{align}
	|\|Sx\|_2^2-\|Sy\|_2^2 -\|Tx\|_2^2 + \|Ty\|_2^2| &=\left|\left\langle S^\top S - T^\top
    T,xx^\top - yy^\top \right\rangle\right|.
  \end{align} 
  By rewriting
  \begin{align}
	2(S^\top S - T^\top T) = (S-T)^\top(S+T) + (S+T)^\top(S-T),
  \end{align}
  \begin{align}
	2(xx^\top - yy^\top)=(x-y)(x+y)^\top + (x+y)(x-y)^\top,
  \end{align}
  and using Cauchy-Schwartz inequality, we get
  \begin{align}
	|\|Sx\|_2^2-\|Sy\|_2^2 -\|Tx\|_2^2 + \|Ty\|_2^2|
	&\leq\|S-T\|_{F}\|S+T\|_{F}\|x-y\|_{2}\|x+y\|_{2}\nonumber\\
	&\leq 4\|S-T\|_{F}\|x-y\|_{2}.
  \end{align} 
  This implies that $f(x) = \|Sx\|_2^2-\|Tx\|_2^2$ is a Lipschitz function on $\mathbb{S}^{d-1}$
  with $\|f\|_{\mathrm{Lip}}\leq 4\|S-T\|_{F}$. Then, by condition \eqref{eq: isoperimetric},
  $\|Sx\|_2^2$ satisfies condition \eqref{eq: subgaussian increment} with
  \begin{align}
	\left\|\|Sx\|_2^2-\|Tx\|_2^2-\mathbb{E}\left(\|Sx\|_2^2-\|Tx\|_2^2\right)\right\|_{\psi_{2}}\leq
    c'\frac{\|S-T\|_{F}}{\sqrt{d}}.
  \end{align} 
  Therefore, \cref{lemma:concentration of residuals} implies
  \begin{align}
	\frac{1}{N}\sup_{\|S\|_{F}=1}\left|\sum_{j=0}^{(1-p)N}\|Sx_i\|^2_{2} -
    \mathbb{E}\|Sx_i\|^2_{2}\right|\leq c\sqrt{\frac{d+\log N}{N}}
	\end{align}
  with probability $1-2N^{-d}$.
\end{proof}

\subsubsection{Proof of \cref{lemma:behav. bad gradient non convex}: Deviation from expectation concerning corrupted points}
\label{sec:behav. bad gradient non convex} 

We now outline the strategy of the proof. We first build an $\varepsilon$-net (\cref{lemma:cover
  SO}) for the set $\mathcal{B}_{\mathrm{skew}}(d):=\{S\mid S=-S^\top, \|S\|_{2}\leq1\}$ and also
$\SO(d)$. Then, we show that the term
\begin{align}
\label{eq: corrupted term}
\sum_{i\in\mathcal{C}}\left\langle RSx_i,\beta_{i}(R) \right\rangle,\quad  R\in\SO(d),\quad S\in\mathcal{B}_{\mathrm{skew}} 
\end{align}
behaves similarly within each neighborhood defined by the $\varepsilon$-nets. Next, for each pair of
$(S,R)$ in $\varepsilon$-nets we bound the difference between \eqref{eq: corrupted term} and its
expectation, which also implies a similar behavior within a neighborhood of the pair $(S,R)$.
Finally, an application of the union bound over all points in $\varepsilon$-net gives a global
deviation bound of \eqref{eq: corrupted term} for all $S,R$ simultaneously. This idea is summarized
in \cref{lemma:bound bad gradient nonconvex}, where \cref{lemma:bound in a patch} gives estimate to
specific term in \cref{lemma:bound bad gradient nonconvex}.


We construct an Euclidean $\varepsilon$-net for $\SO(d)$ and $\mathcal{B}_{\mathrm{skew}}(d)$ using
the following lemma.
\begin{lemma}
	\label{lemma:cover SO}
	Let $\mathcal{N}^{\varepsilon}_{\mathcal{A}}$ be an Euclidean $\varepsilon$-net of 
	\begin{equation}
	  \mathcal{A} := \left\{S\in \mathbb{R}^{d\times d} \ \left| \ \|\mathrm{skew}(S)\|_2 \leq
      \frac{1}{\sqrt{2}},\ S(i,j)=0\ \text{if}\ i\geq j \right. \right\}
	\end{equation}
	then
	\begin{align}\label{eq:def_NB_NS}
	\mathcal{N}^{\epsilon}_{\mathcal{B}} := \left\{\mathrm{skew}(\sqrt{2}\ S)\mid S\in
    \mathcal{N}^{\varepsilon}_{\mathcal{A}}\right\},\quad \mathcal{N}^{\pi\varepsilon}_{\SO} :=
    \left\{\exp\left(\pi \mathrm{skew}(\sqrt{2}\ S)\right)\mid S\in
    \mathcal{N}^{\varepsilon}_{\mathcal{A}}\right\}
	\end{align}
	are Euclidean $\varepsilon$-net of $\mathcal{B}_{\mathrm{skew}}(d)$ and Euclidean
    $(\pi\varepsilon)$-net of $\SO(d)$ respectively, of size
	\begin{align}
	  \label{eq: size e-net}
	  |\mathcal{N}^{\pi \varepsilon}_{\SO}|\leq |\mathcal{N}^{\varepsilon}_{\mathcal{B}}|\leq
      |\mathcal{N}^{\varepsilon}_{\mathcal{A}}|\leq \left(6
      \sqrt{d}\ \varepsilon^{-1}\right)^{\frac{d(d-1)}{2}}.
	\end{align}
	
\end{lemma}
\begin{proof} See \cref{sec: app_enet}
\end{proof}

We construct a cover of $\SO(d)$ by defining for all $Q_l\in \mathcal{N}^{\pi\varepsilon}_{\SO}$
\begin{multline}\label{eq:def_phi}
  \phi_{\pi\varepsilon}(Q_l):=\left\{\exp(\pi\mathrm{skew}(\sqrt{2}A))\vert\right. A\in\mathcal{A}\text{ s.t. }\|A-A_l\|_{F}<\varepsilon,\\ A_l\in  \mathcal{N}_{\mathcal{A}}^{\varepsilon},\ \left. Q_l = \exp\left(\pi \mathrm{skew}(\sqrt{2}\ A_l)\right) \right\}.
\end{multline}
$\phi_{\pi\varepsilon}(\cdot)$ maps the Euclidean $\varepsilon$-ball around $A_l$ in $\mathcal{A}$
to a subset in $\SO(d)$ such that
\begin{align}
  \phi_{\pi\varepsilon}(Q_l)\subseteq\{R\mid\|R-Q_l\|_{F}<\varepsilon\}.
\end{align}
To construct a cover of $\mathcal{B}_{\mathrm{skew}}(d)$, we consider the Euclidean ball
\begin{align}
  \{S\in\mathcal{B}_{\mathrm{skew}}(d)\mid\|S-T_{k}\|_{F}<\pi\varepsilon\}\quad \text{for
    all}\quad T_k\in \mathcal{N}^{\varepsilon}_{\mathcal{B}}.
\end{align}

In the following lemma, we provide an upper bound of \eqref{eq: corrupted term} for all $S\in
\mathcal{B}_{\mathrm{skew}}(d)$ and $R\in\SO(d)$. The upper bound shows two sources of
contribution to the deviation of $\eqref{eq: corrupted term}$:
\begin{enumerate}
	\item The first two terms come from the variation of $\eqref{eq: corrupted term}$ \emph{within each
		part of the partition defined by the $\varepsilon$-nets}.
\item The last term is the deviation of $\eqref{eq: corrupted term}$ from its expectation
  \emph{for points in $\varepsilon$-net}.
\end{enumerate}
Here, for a rotation on the $\varepsilon$-net, instead of having set-valued $\beta_i$, for the sake
of convenience we make it single-valued:
\begin{align}\label{eq:def_beta_til}
  \tilde{\beta}_{i}(R):=\begin{cases}
  \frac{Rx_i-y_i}{\|Rx_i-y_i\|_2}& \text{if}\ Rx_i-y_i\neq 0 \\
  0  & \text{otherwise}.
  \end{cases}
\end{align}. 

\begin{lemma}[Upper bound for \eqref{eq: corrupted term}]\label{lemma:bound bad gradient nonconvex}
  Let $\mathcal{N}^{\varepsilon/\pi}_{\mathcal{A}}$ be an Euclidean $\varepsilon/\pi $-net of $\mathcal{A}$,
  then for all $S\in \mathcal{B}_{\mathrm{skew}}(d)$  and $R\in\SO(d)$, and all choices of $\beta_{i}(R)$,
  \begin{multline}
	\label{eq:bound bad gradient nonconvex}
	\left|\sum_{i\in\mathcal{C}}\left\langle RSx_i,  \beta_{i}(R)\right\rangle\right|\leq\  2 \varepsilon\ pN +\sup_{Q_l\in\mathcal{N}^{\varepsilon}_{\SO}}\sup_{R\in\phi_{\varepsilon}(Q_l)}\sum_{i\in\mathcal{C}} \left\|  \beta_{i}(R) - \tilde{\beta}_{i}(Q_l)\right\|_{2}\\
	+ \sup_{
	  T_{k}\in\mathcal{N}^{\varepsilon/\pi}_{\mathcal{B}},\ 
	  Q_l\in\mathcal{N}^{\varepsilon}_{\SO}}\left|\sum_{i\in\mathcal{C}}\left\langle Q_lT_kx_i,  \tilde{\beta}_{i}(Q_l)\right\rangle\right|.
  \end{multline}
  where $\tilde{\beta}_{i}$'s are defined in \eqref{eq:def_beta_til}.
\end{lemma}

\begin{proof}
  \cref{lemma:cover SO} shows that $\mathcal{N}^{\epsilon/\pi}_{\mathcal{B}}$ is an
  $\varepsilon/\pi$-net of $\mathcal{B}_{\mathrm{skew}}(d)$ and $\mathcal{N}^{\varepsilon}_{\SO}$ is
  an $\varepsilon$-net of $\SO(d)$. Then, for any $S\in \mathcal{B}_{\mathrm{skew}}(d)$ and
  $R\in\SO(d)$, there exists $T_k\in\mathcal{N}^{\epsilon/\pi}_{\mathcal{B}}$ and
  $Q_l\in\mathcal{N}^{\varepsilon}_{\SO}$ such that $\|S-T_k\|_{F}<\varepsilon/\pi$ and
  $R\in\phi_\varepsilon(Q_l)$.
	
  We rewrite $\left|\left\langle Q_lT_kx_i, \tilde{\beta}_{i}(Q_l)\right\rangle - \left\langle
  RSx_i, \beta_{i}(R)\right\rangle\right| $ as
  \begin{align}
	\label{eq: inenet1}
	\frac{1}{2} \left|\left\langle (Q_{l}T_k-RS)x_i,  \tilde{\beta}_{i}(Q_l) +\beta_{i}(R) \right\rangle + \left\langle (Q_{l}T_k+RS)x_i,  \tilde{\beta}_{i}(Q_l) -\beta_{i}(R)\right\rangle\right|,
  \end{align}
  and using Cauchy-Schwartz inequality, we upper bound \eqref{eq: inenet1} by
  \begin{align}
	\label{eq: inenet2}
	\|(Q_{l}T_k-RS)x_i\|_{2} \frac{\|\tilde{\beta}_{i}(Q_l)\|_{2} +\|\beta_{i}(R)\|_{2}}{2}  + \| \tilde{\beta}_{i}(Q_l) -\beta_{i}(R)\|_{2} \frac{\|Q_{l}T_k x_{i}\|_{2}+\|RSx_i\|_{2}}{2}.
  \end{align}
  We have that $\|(Q_{l}T_k-RS)x_i\|_{2}\leq \|Q_{l}-R\|_{F}+\|T_k-S\|_{F}$. Also, by definition of
  $\beta_i(R)$ and $\tilde{\beta}_{i}(R)$, $\|\beta_i(R)\|_{2}\leq 1$ and  $\|\tilde{\beta}_i(R)\|_{2}\leq1$ for any
  $R\in\SO(d)$. Similarly, by definition of $\mathcal{B}_{\mathrm{skew}}$,
  $\|Tx_{i}\|_{2}\leq 1$. Therefore,
  \begin{align}
	\left|\left\langle Q_lT_kx_i, \tilde{ \beta}_{i}(Q_l)\right\rangle - \left\langle RSx_i,
    \beta_{i}(R)\right\rangle\right| &\leq 2\varepsilon + \left\| \tilde{ \beta}_{i}(Q_l)
    -\beta_{i}(R)\right\|_{2}.
  \end{align}
  
  Then for any $R\in\phi(Q_l)$ and $S\in \mathcal{B}_{\mathrm{skew}}(d)$ , we get
  \begin{align}
	\label{eq: bound_bad_gradient_i}
	\left\langle RSx_i, \beta_{i}(R)\right\rangle \leq \left\langle Q_lT_kx_i, \tilde{\beta}_{i}(Q_l)\right\rangle  + 2\varepsilon  + \left\| \tilde{\beta}_{i}(Q_l) -\beta_{i}(R)\right\|_{2}.
  \end{align}	
  Adding \eqref{eq: bound_bad_gradient_i} over all the corrupted points $i\in\mathcal{C}$, and
  taking supremum of the RHS of \eqref{eq: bound_bad_gradient_i} over
  $\mathcal{N}^{\epsilon/\pi}_{\mathcal{B}}$ and $\mathcal{N}^{\varepsilon}_{\SO}$, we get the
  result.
\end{proof}

Next, we provide an upper bound for the last term in the RHS of \eqref{eq:bound bad gradient nonconvex}.
\begin{lemma}\label{lemma:bound in a patch}
  Let $\varepsilon=0.5N^{-\frac{d+2}{2(d-2)}}$, $\mathcal{N}^{\varepsilon}_{\SO}$ the
  ${\varepsilon}$-net of $\SO(d)$ defined in \eqref{eq:def_NB_NS}, $\phi_\varepsilon(\cdot)$ defined
  in \eqref{eq:def_phi}, $\tilde{\beta}_{i}(\cdot)$ defined in \eqref{eq:def_beta_til}, then
	\begin{align}
	\label{eq: bound in a patch}
	\sup_{Q_l\in\mathcal{N}^{\varepsilon}_{\SO}}\sup_{R\in\phi_\varepsilon(Q_l)}\frac{1}{N}\sum_{i\in\mathcal{C}} \left\|  \beta_{i}(R) - \tilde{\beta}_{i}(Q_l)\right\|_{2} &\leq c_{1}d\sqrt{\frac{1}{N}}
	\end{align}
	with probability $1-N^{-c_{2}d^2}$, for $c_{1}$ and $c_{2}$ universal constants.
\end{lemma}

\begin{proof}
  For each $Q_l\in\mathcal{N}^{\varepsilon}_{\SO}$, we construct the index set 
  \begin{align}
	\mathcal{D}_{\delta}(Q_l) := \{i \in \mathcal{C} \vert \ \|Q_lx_i-y_i\|_2\leq \delta\}.
  \end{align}
  If $i\in \mathcal{C}$ does not belong to $\mathcal{D}_{\delta}(Q_l)$, we can bound $ \|
  \beta_{i}(R) - \tilde{\beta}_{i}(Q_l)\|_{2}$ using smoothness of $\beta_i$, as we will see later.
  Therefore, for all $R\in\phi_\varepsilon(Q_l)$, we split the term in the LHS of \eqref{eq: bound in
    a patch} as
  \begin{align}
	\label{eq:full bound 1}
	\sum_{i\in\mathcal{C}} \|\beta_{i}(R) - \tilde{\beta}_{i}(Q_l)\|_{2}
	&\leq 2|\mathcal{D}_{\delta}(Q_l)| + \sum_{i\in\mathcal{C}\setminus\mathcal{D}_{\delta}(Q_l)} \|  \beta_{i}(R) - \tilde{\beta}_{i}(Q_l)\|_{2}
  \end{align}
  using the fact that $\|\beta_{i}(R) - \tilde{\beta}_{i}(Q_l)\|_{2}\leq 2$ for any value of $i$.
  
  Consider $\mathbbm{1}_{\mathcal{D}_{\delta}(Q_l)}(i)$ the indicator function of
  $i\in\mathcal{D}_{\delta}(Q_l)$. Let $x,y\sim\mathrm{Unif}(\mathbb{S}^{d-1})$. Then, rotation
  invariance and \cref{lemma:exp_x_y sphere} imply that for all $i\in\mathcal{C}$,
  $\mathbbm{1}_{\mathcal{D}_{\delta}(Q_l)}(i)$ is a Bernoulli random variable with probability
  \begin{align}\label{eq:p_bern_D}
	p_{\delta}:=P(\|x-y\|_{2}<\delta) \leq \frac{\delta^{d}}{5 \sqrt d}.
  \end{align}
  Therefore, $\mathbbm{1}_{\mathcal{D}_{\delta}(Q_l)}(i)$ is a subgaussian random variable \eqref{eq:subgaussian norm bernoulli} with norm
  \begin{eqnarray}
	\|\mathbbm{1}_{\mathcal{D}_{\delta}(Q_l)}(i)\|^2_{\psi_{2}}\leq (2\log(5\sqrt{d} \delta^{-d}/2))^{-1}.
  \end{eqnarray}
  Since $|\mathcal{D}_{\delta}(Q_l)|
  =\sum_{i\in\mathcal{C}}\mathbbm{1}_{\mathcal{D}_{\delta}(Q_l)}(i)$, then by Hoeffding's
  inequality (\cref{theorem:hoeffding}), for any $t>0$
  \begin{equation}
	\label{eq:blow up bound}
	\frac{\vert \mathcal{D}_{\delta}(Q_l)\vert}{N} \leq \frac{\delta^{d}}{5 \sqrt d}+ t
    \ \text{ with probability }\  1-\exp\left(-\frac{c N t^2}{p}\log\left(\frac{5 \sqrt
      d}{2\delta^{d}}\right)\right)
  \end{equation}
  for some universal constant $c>0$. Notice that this bound only depends on $Q_l$ and not on other
  $R\in\phi_\varepsilon(Q_l)$.
  
  Now, we want to estimate the second term of the RHS of \eqref{eq:full bound 1}. First, for all
  $i\in\mathcal{C}\setminus\mathcal{D}_{\delta}(Q_l)$, $ \|Q_lx_i-y_i\|_2> \delta$.  Then, for
  $\delta>\varepsilon$, any $R\in\phi_\varepsilon(Q_l)$ satisfies
  \begin{align}
	\|Rx_i - y_i\|_2 \geq\|Q_lx_i-y_i\|_2 - \|(R-Q)x_i\|_2 \geq \delta -\varepsilon>0.
  \end{align}
  Therefore 
  \begin{align}
	\beta_{i}(R) = \frac{R x_{i}-y_i}{\|R x_{i}-y_i\|_{2}}\quad\text{and}\quad 	\tilde{\beta}_{i}(Q_l) = \frac{Q_l x_{i}-y_i}{\|Q_l x_{i}-y_i\|_{2}}.
  \end{align}
  Moreover, 
  \begin{multline}
	\label{eq:non-blow up bound}
	2\|\beta_{i}(R)-\tilde{\beta}_{i}(Q_l)\|_{2}\leq \left \vert \| Rx_i - y_i \|_2^{-1}+ \| Q_lx_i
    - y_i \|_2^{-1}\right \vert \|(R-Q_l)x_i\|_2\\ +\left \vert \| Rx_i - y_i \|_2^{-1}- \| Q_lx_i -
      y_i \|_2^{-1}\right \vert\|(R+Q_l)x_i-2y_i\|_2.
  \end{multline}
  Since $\| Rx_i - y_i \|_2^{-1}$ is differentiable in $\phi_\varepsilon(Q_l)$, then
  \begin{align}
	\left\vert\| Rx_i - y_i \|_2^{-1}-\| Q_lx_i - y_i \|_2^{-1}\right\vert\leq
    \varepsilon\max_{\tilde{R}\in \phi_\varepsilon(Q_l)}\left \|\frac{\tilde{R}x_i-y_i}{\|
      \tilde{R}x_i - y_i\|^3_2}\right\|_2.
  \end{align}
  Hence, for all $i\notin\mathcal{D}_{\delta}(Q_l)$ and $R\in\phi_\varepsilon(Q_l)$, we get the
  bound
  \begin{align}
	\|\beta_{i}(R)-\beta_{i}(\tilde{Q}_l)\|_{2}\leq\frac{\varepsilon(\delta-\varepsilon + 2)}{(\delta-\varepsilon)^{2}}.
  \end{align}
  
  \cref{lemma:cover SO} implies $|\mathcal{N}^{\varepsilon}_{\SO}|\leq (6\sqrt{d}\pi
  \varepsilon^{-1})^{d(d-1)/2}$.  In particular, taking $\varepsilon = N^{-(d+2)/(2(d-2))}$,
  $\delta = (2\varepsilon)^{1/(d+2)}$ and $t= d\sqrt{1.5 N^{-1} c^{-1}}$, by union bound over all
  $Q_l\in\mathcal{N}^{\varepsilon}_{\SO}$, we have that
  \begin{align}
	\sup_{Q_l\in\mathcal{N}^{\varepsilon}_{\SO}}\sup_{R\in\phi_\varepsilon(Q_l)}\frac{1}{N}\sum_{i\in\mathcal{C}}
    \left\| \beta_{i}(R)-\tilde{\beta}_{i}(Q_l)\right\|_{2} &\leq c_{1}d\sqrt{\frac{1}{N}}
  \end{align}
  with probability $1-N^{-c_{2}d^2}$, for $c_{1}$ and $c_{2}$ universal constants.
\end{proof}

We finally conclude this section with a proof of \cref{lemma:behav. bad gradient non convex}.
  \begin{proof}[Proof of \cref{lemma:behav. bad gradient non convex}]
  To control the second term in the RHS of \eqref{eq:bound bad gradient nonconvex}, we use the fact
  that for a given $Q_l\in\mathcal{N}^{\varepsilon}_{\SO}$ and
  $T_k\in\mathcal{N}^{\varepsilon/\pi}_{\mathcal{B}}$,
  \begin{align}
	\left\langle Q_lT_kx_i, \tilde{\beta}_{i}(Q_l)\right\rangle = \left\langle Q_lT_kx_i, \frac{Q_l
      x_i - y_i}{\|Q_l x_i - y_i\|_{2}}\right\rangle\ \text{ almost everywhere for }i\in\mathcal{C}.
  \end{align}
  
  \cref{lemma: MGF xS(x-y)} shows that for any $S\in\mathcal{B}_{\mathrm{skew}}(d)$ and
  $x,y\sim\mathrm{Unif}(\mathbb{S}^{d-1})$ independent,
  \begin{align}
	\E\left[\exp\left(\lambda \left\langle Sx,\frac{x-y}{\|x-y\|_{2}}\right\rangle\right)\right]\leq
    \exp\left(\frac{\lambda^2}{4(d-1)}\right).
  \end{align}
  Then, rotation invariance and \eqref{eq: MGF subgaussian} imply
  \begin{align}
	\left\|\left\langle Q_lT_kx_i, \tilde{\beta}_i(Q_l)\right\rangle\right\|^2_{\psi_2}\leq
    c\left(4(d-1)\right)^{-1},\text{ for all }i\in\mathcal{C}.
  \end{align}
  Using Hoeffding's inequality (\cref{theorem:hoeffding}), we get
  \begin{align}
	\left|\sum_{i\in\mathcal{C}}\left\langle Q_lT_kx_i,
    \tilde{\beta}_{i}(Q_l)\right\rangle\right|\leq t\quad\text{with probability}\quad
    1-2\exp\left(-\frac{c_2(d-1) t^2}{N\ p}\right).
  \end{align}
  
  \cref{lemma:cover SO} implies $ |\mathcal{N}^{\varepsilon}_{\SO}|\leq
  |\mathcal{N}^{\varepsilon/\pi}_{\mathcal{B}}|\leq (6\sqrt{d}\pi \varepsilon^{-1})^{d(d-1)/2}$.
   Let $\varepsilon=N^{-(d+2)/(2(d-2))}$ and $t=2\sqrt{d\log(N)c_2^{-1} N}$. By union bound over all
  $T_k\in\mathcal{N}^{\varepsilon}_{\SO}$ and $Q_l\in\mathcal{N}^{\varepsilon}_{\SO}$, we get
  \begin{align}
	\label{eq: union bound bad grad}
	\sup_{T_{k}\in\mathcal{N}_{1},Q_l\in\mathcal{N}_{2}}\left|\sum_{i\in\mathcal{C}}\left\langle
    Q_lT_kx_i, \tilde{\beta}_{i}(Q_l)\right\rangle\right|\leq c \sqrt{\frac{d\log(N)}{N}}
  \end{align}
  with probability $1-N^{-2d^2}$.  
	
  From \cref{lemma:bound bad gradient nonconvex} we have that for any choice of $\beta_{i}(R)$,
  \begin{multline}
	\sup_{R\in\SO(d),S\in\mathcal{B}_{\mathrm{skew}}}\left|\sum_{i\in\mathcal{C}}\left\langle RSx_i,
    \beta_{i}(R)\right\rangle\right|\leq 2 \varepsilon\ pN + \sup_{
      T_{k}\in\mathcal{N}^{\varepsilon/\pi}_{\mathcal{B}},
      Q_l\in\mathcal{N}^{\varepsilon}_{\SO}}\left|\sum_{i\in\mathcal{C}}\left\langle Q_lT_kx_i,
    \tilde{\beta}_{i}(Q_l)\right\rangle\right|\\
    +\sup_{Q_l\in\mathcal{N}^{\varepsilon}_{\SO}}\sup_{R\in\phi(Q_l)}\sum_{i\in\mathcal{C}} \left\|
    \beta_{i}(R) - \tilde{\beta}_{i}(Q_l)\right\|_{2}
  \end{multline}
  Inserting \eqref{eq: union bound bad grad} and the upper bound \eqref{eq: bound in a patch} in
  \cref{lemma:bound in a patch}, we get that with high probability, for all choices of
  $\beta_{i}(R)$,
  \begin{align}
    \sup_{R\in\SO(d),S\in\mathcal{B}_{\mathrm{skew}}}
    \left|\frac{1}{N}\sum_{i\in\mathcal{C}}\left[\left\langle RSx_i, \beta_{i}(R)\right\rangle -
      \mathbb{E}\left\langle RSx_i, \beta_{i}(R)\right\rangle\right]\right|&\leq c \sqrt{\frac{d(\log
        (N)+d)}{N}}.
  \end{align}
\end{proof}

\section{Numerical Simulations}\label{sec:Num.Sim.}

We compare the results of \cref{thm:sharpness convex} and \cref{thm:imp. recovery} with
numerical solution of \cref{prob:conv2}, and the results of \cref{thm:converg. nonconvex} with the
solution of \cref{prob:LUD}. To solve \cref{prob:conv2}, we use CVX, a package for specifying and
solving convex programs \cite{cvx},\cite{gb08}. We enforce the $\mathrm{conv}\,\SO(d)$ constraint
using the result \cite[Theorem 1.3]{Saunderson_2015}
\begin{align}
\label{def: conv SO(d)}
\mathrm{conv}\, \SO(d)=\left\{X\in\mathbb{R}^{n} : \left[\begin{array}{cc}
0 & X\\
X^\top & 0
\end{array}\right] \preceq I_{2n}, \sum_{i,j=1}^{n} A^{(i,j)} \left[DX\right]_{ij}\preceq (n-2) I_{2^{n-1}}\right\},
\end{align}
where $D:=\mathrm{diag}(1,\dots,1,-1)$, $A^{(i,j)} :=-P_{\mathrm{even}}^\top \lambda_i \rho_j
-P_{\mathrm{even}}$, $\lambda_i := D_{2}^{\otimes i-1}\otimes A_2 \otimes I_{2}^{\otimes n-i}$,
$\rho_i := I_{2}^{\otimes i-1}\otimes A_2 \otimes D_{2}^{\otimes n-i}$ and
\begin{align}
P_{\mathrm{even}} :=\frac{1}{2}\left[\begin{array}{c}
1\\1
\end{array}\right]\otimes I_2^{\otimes n-1} +  \frac{1}{2}\left[\begin{array}{c}
1\\-1
\end{array}\right]\otimes D_{2}^{\otimes n-1} \quad \text{for} \quad D_{2} := \left[\begin{array}{cc}
1&0\\
0 & -1
\end{array}\right].
\end{align} 
To solve \cref{prob:LUD}, we implement a subgradient descent method, which can be viewed as a
discretization of the dynamical system \eqref{eq:dyn. system}, using line search as specified in
\cite[Chapter 3.5]{nocedal_2006}. \cref{alg:non_conv} shows a pseudocode of the implementation. The
length of each step is controlled by $\alpha_{\max}$. To optimize over the manifold $SO(d)$, the
line search is done over the geodesic defined by the Riemannian gradient of $L$ at $R$,
$\partial_RL$.
\begin{algorithm}[htb]
	\begin{small}
		\begin{center}
			\begin{minipage}{0.99\textwidth}
				\begin{algorithmic}[1]
					\Procedure {OptimalRotation}{$\{x\}_{i=1}^{N}$ ,$\{y\}_{i=1}^{N}$,$R^0\in\SO(d)$, $\mathrm{tol}>0$,$\alpha_{\max}$}
					\For {${k}\gets 1,\dots,\mathrm{niter}$}
					\State $\partial L^{k} \gets 0$ 
					\For{$i\gets 1,\dots, N$}
					\State $u^{k}_{i}\gets R^{k-1}x_{i}- y_{i}$ 
					\State $u_{\mathrm{eps}} \gets 0$
					\If {$\|u^{k}_i\|_{2}<\mathrm{tol}$}
					\State $u_{\mathrm{eps}}=\mathrm{tol}$
					\EndIf
					\State $u^{k}_i\gets u^{k}_i/(u_{\mathrm{eps}} + \|u^{k}_i\|_{2})$
					\State $\partial L^{k}\gets dL^{k} + u^{k}_i x_i^\top$  
					\EndFor
					\State $\partial_{R}L^{k} \gets -((R^{k-1})^\top \partial L^{k} - (\partial L^{k})^\top R^{k-1})/(2\ N) $
					\State $\alpha_k\gets$\Call{Line-Search}{$f(\alpha)= L(R^{k-1}\exp(-\alpha\cdot \partial_{R}L^{k});\{x\}_{i=1}^{N} ,\{y\}_{i=1}^{N}),\alpha_{\max}$}
					\State $R^{k}=R^{k-1}\exp(-\alpha_{k}\cdot \partial_{R}L^{k})$
					\If {$\|R^{k}-R^{k-1}\|_{F}<\mathrm{tol}$} 
					\State \textbf{break}
					\EndIf
					\EndFor
					\EndProcedure
				\end{algorithmic}
			\end{minipage}
		\end{center}
	\end{small}
	\caption{Line-Search method over $\SO(d)$.}%
	\label{alg:non_conv}%
\end{algorithm}

\subsection{Uniformly distributed  in $\mathbb{S}^{d-1}$}

In this scenario, given a sample size $N$ and corruption level $p$, we generate
$\{x_i\}_{i=1}^{N}$, $\{y_i\}_{i=1}^{N}$ following the distribution specified in
\eqref{eq:prob_model0} and \eqref{eq:prob_model}. We assume the ground truth rotation $R_0=I$.

\subsubsection{Trade-off between corruption level and sample size}\label{sec:p_vs_N}
We first study the impact of the corruption level and the sample size in recovering the ground
truth.  We consider sample sizes $N\in[4,1024]$ and corruption levels $p \in [0.1,0.99]$. We
generate 10 independent random samples of $\{x_i\}_{i=1}^{N}$, $\{y_i\}_{i=1}^{N}$ for each choice
of $N$ and $p$. For each random sample we minimize $L(A)$ over $\mathrm{conv}\ \SO(d)$ and
$\SO(d)$, taking as initial point the solution of the least squares problem. We say we recover the
ground truth if $\|I-A^{*}\|\leq 10^{-2}$, where $A^{*}$ is the minimizer. We then compute the
empirical probability of exact recovery for each combination of $p,N$, as shown in
\cref{fig:compare errors}.

\begin{figure}[h!]
  \centering

  \subfloat[$\mathrm{conv}\,\SO(3)$\label{fig:error_cvxSO3}]{
  	\centering
  	\includegraphics{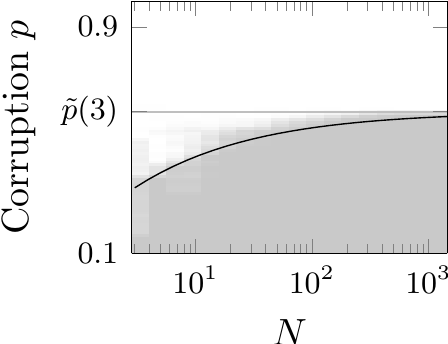}
	}
  \subfloat[][$\mathrm{conv}\,\SO(3)$\\ projected on $\SO(3)$\label{fig:error_cvxprojSO3}]{
  	  	\centering
  	\includegraphics{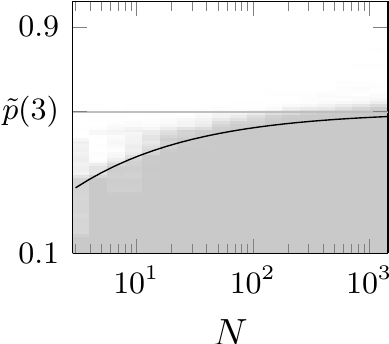}
	}
  \subfloat[$\SO(3)$\label{fig:error_noncvxSO3}]{
  	  	\centering
  	\includegraphics{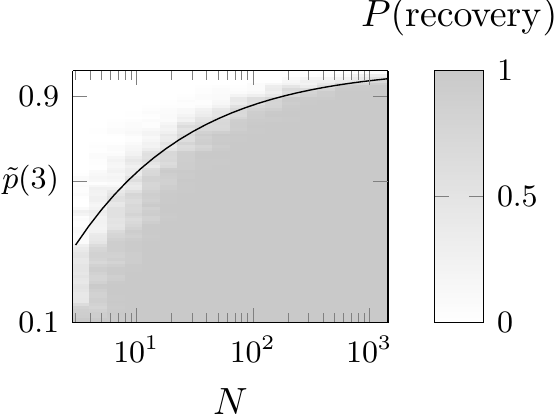}
   }
  \newline
  \subfloat[$\mathrm{conv}\,\SO(6)$\label{fig:error_cvxSO6}]{
  	  	\centering
  	\includegraphics{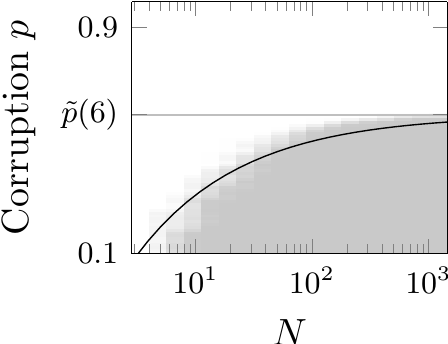}
   }
  \subfloat[][$\mathrm{conv}\,\SO(6)$\\projected on $\SO(6)$\label{fig:error_cvxprojSO6}]{
  	  	\centering
  	\includegraphics{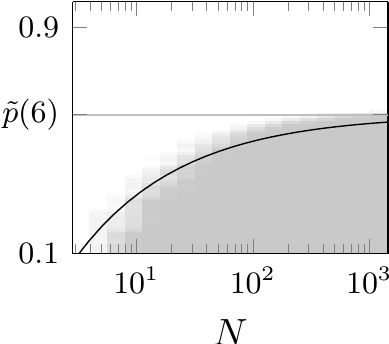}
	}
  \subfloat[$\SO(6)$\label{fig:error_noncvxSO6}]{
  	  	\centering
  	\includegraphics{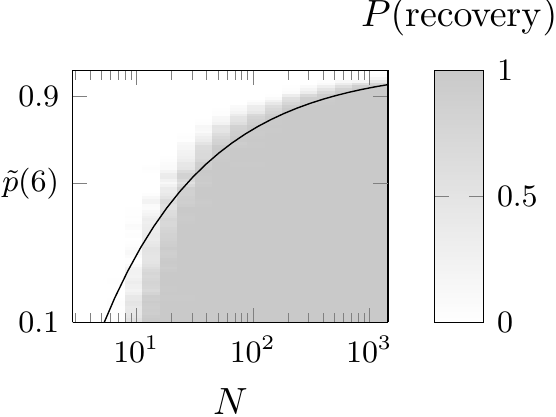}
	}

  \caption{Empirical probability of exact recovery minimizing $L(A)$ over different domains. Data is
    uniformly distributed over $\mathbb{S}^{d-1}$. The dark line is the theoretical upper bound of
    admissible corruption level in \eqref{eq:possible_recovery} and \eqref{eq:
      cond_p_non_cvx}.}\label{fig:compare errors}
\end{figure}

We start by analyzing the probability of exact recovery solving \cref{prob:conv2}, i.e. minimizing
$L(A)$ over $\mathrm{conv}\, \SO(d)$. \cref{fig:error_cvxSO3} and \cref{fig:error_cvxSO6} show that
the transition between exact recovery in all of the experiments and no recovery in any of them
follows the bound \eqref{eq:possible_recovery} of \cref{thm:sharpness convex}, equals to
$\tilde{p}(d)-c\sqrt{\log N/N}$.
 
Since the minimizer of $L(A)$ over $\mathrm{conv}\, \SO(d)$, $A^{*}_{\mathrm{conv}\, \SO(d)}$, is
not necessarily a rotation, one can project $A^{*}_{\mathrm{conv}\, \SO(d)}$ on $\SO(d)$ to find the
closest rotation. Although the theorems in this paper do not provide any insight in this case,
empirically \cref{fig:error_cvxprojSO3} and \cref{fig:error_cvxprojSO6} show that projecting the
solution of \cref{prob:conv2} over $\SO(d)$ does not significantly improve the probability of exact
recovery when the corruption level is above $\tilde{p}(d)$.
 
Then, we inspect the probability of exact recovery solving \cref{prob:LUD}, i.e. minimizing $L(A)$
over $SO(d)$. \cref{fig:error_noncvxSO3} and \cref{fig:error_noncvxSO6} show the transition between
exact recovery in all of the experiments and no recovery in any of them follows the bound \eqref{eq:
  cond_p_non_cvx}, derived from \cref{thm:converg. nonconvex} and equals to $p<1-c\sqrt{\log
  N/N}$. This means that we can always recover the ground truth regardless of the corruption level
if we have enough samples.

\subsubsection{Impact of initialization in dynamical system \eqref{eq:dyn. system}}\label{sec:impact_r0}
Exact recovery of the ground truth rotation minimizing $L(A)$ over $SO(d)$ depends on the selection
of the initial point in the dynamical system \eqref{eq:dyn. system}. Our second experiment explores
the influence of this selection. We consider four sample sizes, $N\in\{24,32,64,128\}$. For each
$N$, we generate a random sample $\{x_i\}_{i=1}^{N}$, $\{y_i\}_{i=1}^{N}$ with $d=4$ and
$p=0.75$. Then, for each sample, we solve the dynamical system \eqref{eq:dyn. system} a hundred
times, each time starting from a rotation $R(0)$ chosen at random such that $\|\log(R(0))\|_{2}$ is
uniformly distributed in $[0,\pi)$.  Given a starting point $R(0)$, if the solution of the dynamical
  system $R(t)$ converges to $I$, we record the time $T_{\mathrm{cvg}}$ such that
  $\|\log(R(T{\mathrm{cvg}}))\|_{2}<10^{-2}$, i.e. $D_{\SO}(R(T_{\mathrm{cvg}}),I)$ is
  small. Otherwise, we include $R(0)$ in the set of starting points without exact recovery. We also
  consider all the points in the trajectory $R(t)$, and, for each of them, we compute
  $T_{\mathrm{cvg}}$ as if the dynamical system started from there.
  
\begin{figure}[h!]
	\centering
	\subfloat[$N=24$\label{fig:traj_24}]{
		\centering
		\includegraphics{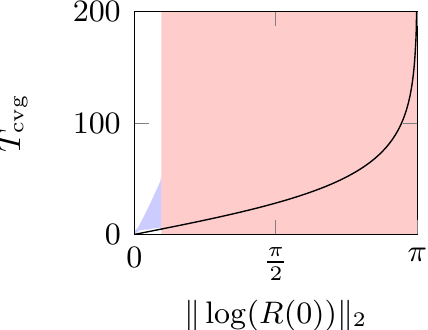}
		}
	\subfloat[$N=32$\label{fig:traj_32}]{
		\centering
		\includegraphics{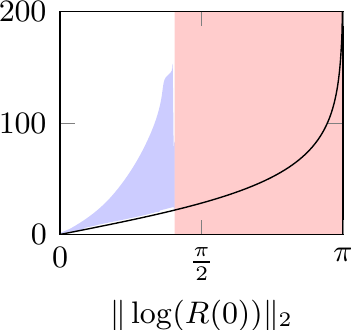}
		}
	\subfloat[$N=64$\label{fig:traj_64}]{
		\centering
		\includegraphics{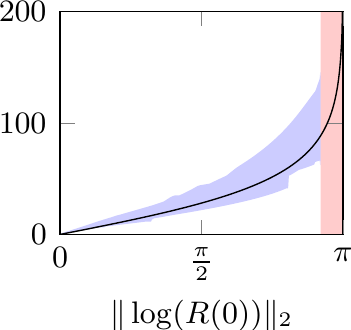}
		}
	\subfloat[$N=128$\label{fig:traj_128}]{
		\centering
		\includegraphics{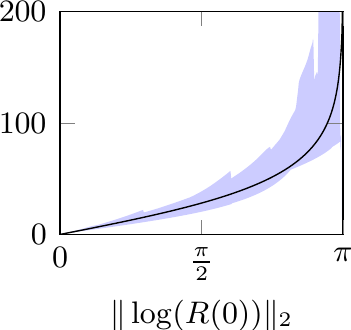}
		}
	\caption{Empirical envelop of convergence time $T_{cvg}$ minimizing $L(A)$ over $\SO(4)$ with
      initial point $R(0)$. The data is distributed uniformly in $\mathbb{S}^{3}$ with corruption
      level $p=0.75$. The red region corresponds to values of $\|\log(R(0))\|_{2}$ where exact
      recovery is not guaranteed. The dark line is the theoretical bound of convergence time
      \eqref{eq:conv rate} in \cref{thm:converg. nonconvex}.}\label{fig:compare trajectory}
\end{figure}
 
\cref{fig:compare trajectory} shows in blue an envelop of the convergence time to $I$,
$T_{\mathrm{cvg}}$, in terms of $\|\log R(0)\|_{2}$, where $R(0)$ is the starting point of the
dynamical system \eqref{eq:dyn. system}. If empirically we observe a dynamical system with starting
point $R(0)$ do not converge to $I$, then all the values greater or equal than $\|\log R(0)\|_{2}$
belong to the region of no exact recovery, denoted in red.

Comparing \cref{fig:traj_24} through \cref{fig:traj_128}, we notice the region of exact recovery
increases as the sample size $N$ gets larger, as described in condition \eqref{eq: in_R_dyn} of
\cref{thm:converg. nonconvex}. Moreover, \cref{fig:traj_128} shows that the convergence time
$T_{\mathrm{cvg}}$ follows the curve $T(\|\log(R(0))\|_{2}) =
\frac{d}{1-p}\cosh^{-1}\left(\sec\left(\frac{\|\log(R(0))\|_{2}}{2}\right)\right) +o(1)$ equals to
the convergence time \eqref{eq:conv rate} in \cref{thm:converg. nonconvex}.

\subsection{Stanford Bunny}

Although the theorems proven in this paper assume that the data is uniformly distributed on
$\mathbb{S}^{d-1}$, we are interested in the generalization of the bounds for other data
distributions. For example, a common source of corruption is mislabeling of data pairs.  We create a
data set using the  CT scan of the
Stanford terra-cota bunny \cite{stanfordBunny}, \cite{stanfordBunnyCT}. We select the points with
highest intensity, and we construct the set of coordinates
$\{b_{j}\}_{j=1}^{N_{Bn}}\subset\mathbb{R}^{3}$ where $N_{Bn}\approx230\mathrm{K}$ entries. We
normalize the points to be centered at zero and to have a maximum length of one. This point cloud
can be modeled as the discrete distribution $\rho_{Bn} (x):=
\frac{1}{N_{Bn}}\sum_{j=1}^{N_{Bn}}{\delta_{b_j}(x)}$. For a given sample size $N$, we generate
samples $\{x_i\}_{i=1}^{N}, \{y_i\}_{i=1}^{N}$ as specified in \eqref{eq:prob_model0} and
\eqref{eq:prob_model}, except now we change from uniform distribution over a sphere to $\rho_{Bn}
(x)$.
\begin{figure}[h!]
  \centering
  \includegraphics[width=0.6\textwidth]{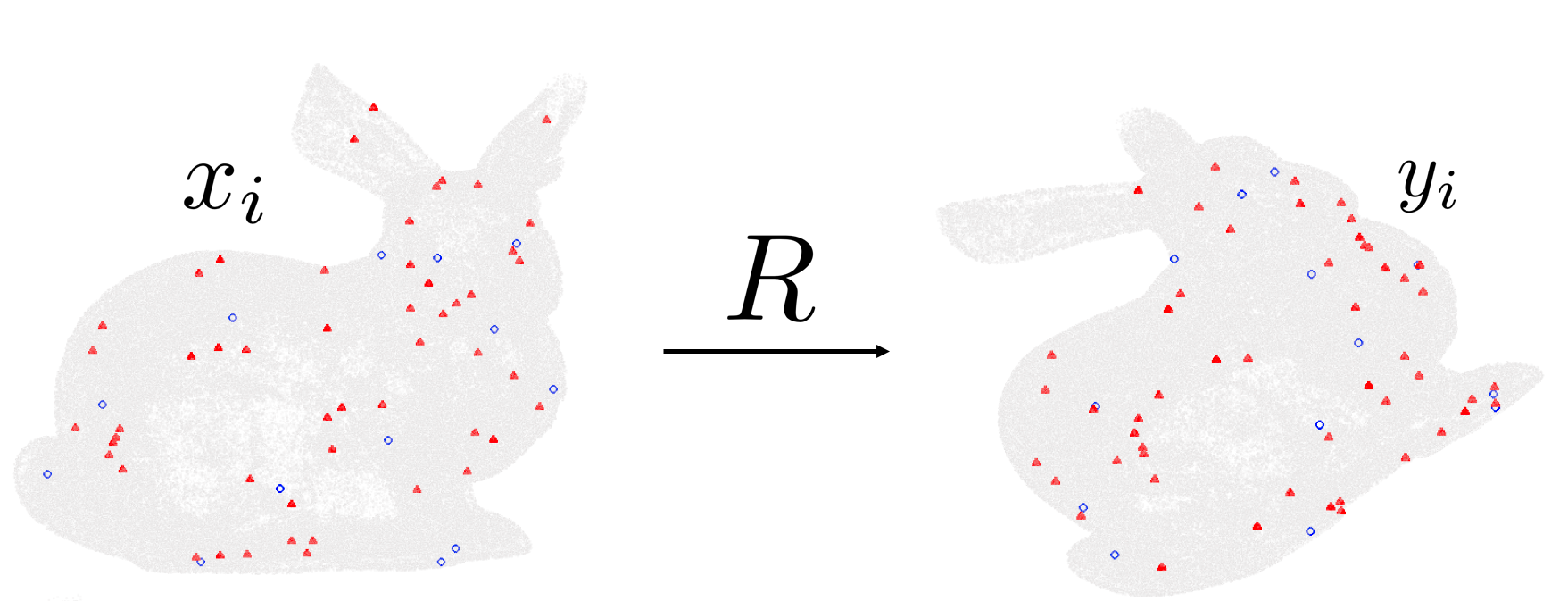}
 	\caption{Data generation using the Stanford Bunny. Uncorrupted points are denoted by blue
      circles, whereas corrupted ones are red triangles. } \label{fig:generating_bunny}
\end{figure}

\subsubsection{Trade-off between corruption level and sample size}\label{sec:bun_p_vs_N}
In \cref{fig:compare errors_bunny}, we reproduce the experiment in \cref{sec:p_vs_N}, except now samples are  
drawn from $\rho_{Bn}(x)$.  Overall, there are not noticeable changes between \cref{fig:compare
  errors} and \cref{fig:compare errors_bunny}.

\begin{figure}[h!]
  \centering
\subfloat[$\mathrm{conv}\,\SO(3)$\label{fig:error_cvx_bun}]{
	\centering
	\includegraphics{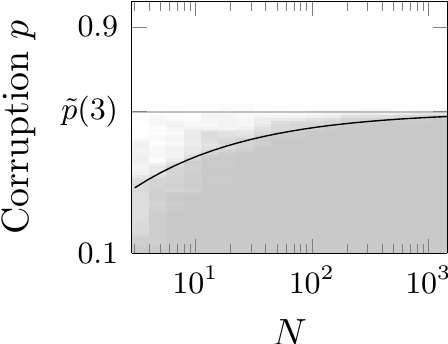}
	}
\subfloat[][$\mathrm{conv}\,\SO(3)$\\projected on $\SO(3)$\label{fig:error_cvxproj_bun}]{
	\centering
	\includegraphics{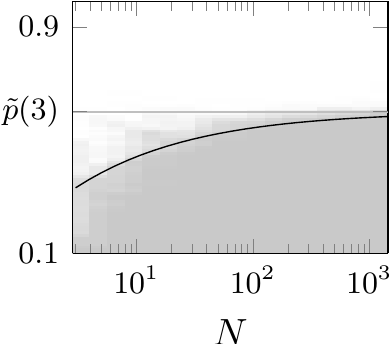}
	}
\subfloat[$\SO(3)$\label{fig:error_noncvx_bun}]{
	\centering
	\includegraphics{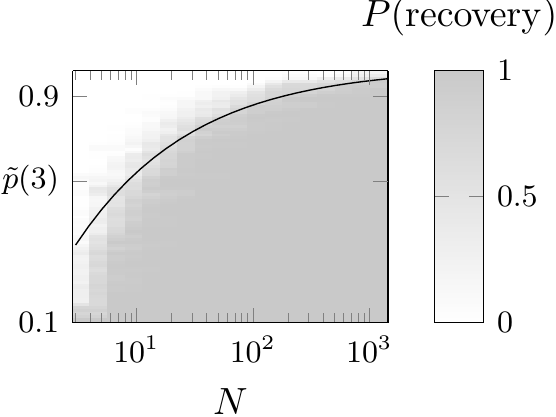}
	}
	\caption{Empirical probability of exact recovery minimizing $L(A)$ over different domains. Data
      is uniformly distributed over Stanford Bunny, $\rho_{Bn}(x)$. The dark line is the theoretical
      upper bound of admissible corruption level in \eqref{eq:possible_recovery} and \eqref{eq:
        cond_p_non_cvx}.}
	\label{fig:compare errors_bunny}
\end{figure}

We start by considering the solution of \cref{prob:conv2}, i.e. optimizing $L(A)$ over
$\mathrm{conv}\,SO(3)$. \cref{fig:error_cvx_bun} shows there is no exact recovery in any of the
experiments when $p>\tilde{p}(3)$ . Additionally the level set of exact recovery in all of the
experiments follows the curve $\tilde{p}(3)-c\sqrt{\log N/N}$, equals to the bound
\eqref{eq:possible_recovery} of \cref{thm:sharpness convex}. Similar to \cref{fig:error_cvxprojSO3},
\cref{fig:error_cvxproj_bun} shows there is not substantial improvement in the probability of exact
recovery when the solution of \cref{prob:conv2} is projected on $\SO(3)$ to get back a rotation
matrix.

Regarding the minimization of $L(A)$ over $\SO(3)$, \cref{fig:error_noncvx_bun} shows the transition
transition between exact recovery in all experiments and no recovery in any experiment follows the
curve $p < 1- c\sqrt{\log N/N}$, equal to the bound \eqref{eq: cond_p_non_cvx}, derived from
\cref{thm:converg. nonconvex}. Similar as the data uniformly distributed over $\mathbb{S}^{2}$, we can
always recover the ground truth regardless of the corruption level if we have enough samples.

\subsubsection{Impact of initialization in dynamical system \eqref{eq:dyn. system}}\label{sec:impact_r0_bun}
In \cref{fig:compare trajectory_bunny}, we reproduce the experiment in \cref{sec:impact_r0}, except
now sampling from $\rho_{Bn}(x)$. Overall, \cref{fig:compare trajectory_bunny} follows the same
behavior as \cref{fig:compare trajectory}. First, the region where exact recovery is not guaranteed
shrinks as the sample size grows, as shown by \cref{fig:traj_24_bun} through
\cref{fig:traj_128_bun}. Second, in \cref{fig:traj_64_bun} and \cref{fig:traj_128_bun} we can see
that the convergence time $T_{\mathrm{cvg}}$ follows the curve $T(\|\log(R(0))\|_{2}) =
\frac{d}{1-p}\cosh^{-1}\left(\sec\left(\frac{\|\log(R(0))\|_{2}}{2}\right)\right) +o(1)$ equals to
the convergence time \eqref{eq:conv rate} in \cref{thm:converg. nonconvex}.
\begin{figure}[h!]
   \centering
	\subfloat[$N=24$\label{fig:traj_24_bun}]{
		\centering
		\includegraphics{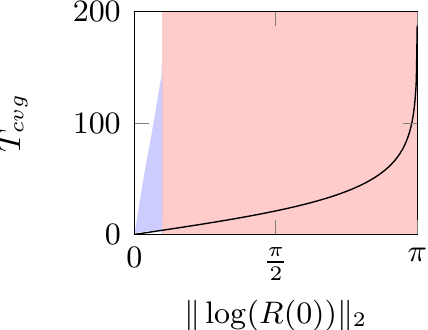}
	}
	\subfloat[$N=32$\label{fig:traj_32_bun}]{
		\centering
		\includegraphics{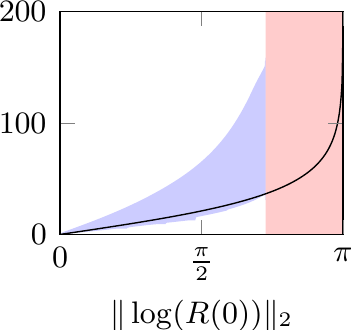}
	}
	\subfloat[$N=64$\label{fig:traj_64_bun}]{
		\centering
		\includegraphics{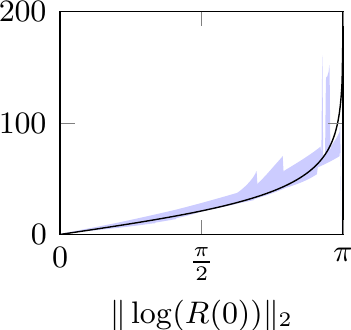}
	}
	\subfloat[$N=128$\label{fig:traj_128_bun}]{
		\centering
		\includegraphics{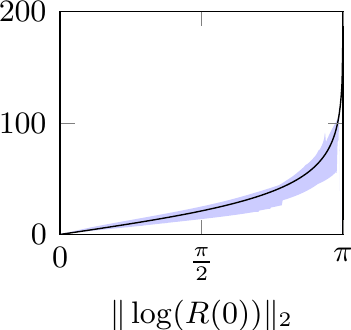}
	}
 	\caption{Empirical envelop of convergence time $T_{cvg}$ minimizing $L(A)$ over $\SO(3)$ with
      initial point $R(0)$. The data is distributed uniformly over Stanford Bunny, $\rho_{Bn}(x)$,
      with corruption level $p=0.75$. The red region corresponds to values of
      $\|\log(R(0))\|_{2}\|)$ where exact recovery is not guaranteed. The dark line is the
      theoretical bound of convergence time \eqref{eq:conv rate} in
      \cref{thm:converg. nonconvex}.} \label{fig:compare trajectory_bunny}
\end{figure}

\section{Conclusion}

We prove the point-set registration problem with outliers can be exactly solved by minimizing the
least-unsquared-deviaiton (LUD) over $\mathbb{R}^{d}$ or $\mathrm{conv}\,\SO(d)$ only when the corruption level $p$ is
less than $\tilde{p}(d)-o(1)$. On the other hand, we proved that we can exactly recover the
ground truth rotation $R_0$ by minimizing the LUD over $\SO(d)$ using the gradient flow
\eqref{eq:dyn. system} for any corruption level $p<1$ and initial point $R(0)$ with
$\|\log(R_0^\top R(0))\|_{2}<\pi$ when the sample size $N$ is large enough. We showed these results are
consistent with numerical simulations for data uniformly distributed on $\mathbb{S}^{d-1}$ and on
discrete points of Stanford bunny. In future work, we shall explore the extension of this
theoretical bounds to arbitrary distributions.

\section*{Acknowledgments}
The work of C. O. is partially supported by the Stanford Graduate Fellowship in Science \&
Engineering. The work of L.Y. is partially supported by the National Science Foundation under award
DMS-1818449 and by the U.S. Department of Energy, Office of Science, Office of Advanced Scientific
Computing Research, Scientific Discovery through Advanced Computing (SciDAC) program.

\appendix
\section{Bounds on sub-gaussian norms for particular random variables}
In this section, we provide some common examples of sub-gaussian random variables.
\begin{enumerate}
	\item \textit{Bounded random variables} \cite[Chapter~2]{vershynin_2018} any bounded random
      variable $X$ is sub-gaussian with
	\begin{align}
	\label{eq: bounded subgaussian norm}
	\|X\|_{\psi_2}\leq \frac{\|X\|_{\infty}}{\sqrt{\log 2}}
	\end{align} 
	\item \textit{Bernoulli random variables} \cite{Buldygin_2012} any $X$ Bernoulli random variable
      with $p\in(0,1/2)$ is sub-gaussian with
	\begin{align}
	\label{eq:subgaussian norm bernoulli}
	\|X-\E X\|^2_{\psi_2}=\frac{1-2p}{2(\log(1-p)-\log(p))}\leq \frac{1}{2|\log\left(2p\right)|}
	\end{align}
	\item \textit{Lipschitz function on the sphere} \cite[Chapter~5]{vershynin_2018} Let $f:
	\mathbb{S}^{d-1}\rightarrow \mathbb{R}$ be a Lipschitz function with Lipschitz constant
	$\|f\|_\text{Lip}$. If $X\sim \text{Unif}(\mathbb{S}^{d-1})$, then $f(X) - \mathbb{E}f(X)$ is
	sub-gaussian with
	\begin{equation}
	\label{eq: isoperimetric}
	\|  f(X) - \mathbb{E}f(X) \|_{\psi_2} \leq \frac{c \|f\|_\text{Lip} }{\sqrt d}.
	\end{equation}
	for some universal constant $c>0$.
\end{enumerate}

\section{Properties of $\mathrm{Unif}(\mathbb{S}^{d-1})$}\label{sec:prop_sphere}
In this section we list useful properties related with data uniformly distributed on the sphere on dimension $d$, $\mathbb{S}^{d-1}$. To start, we use the fact that for
$x\in\mathrm{Unif}(\mathbb{S}^{d-1})$
\begin{align}
\label{eq: exp cov sphere}
\mathbb{E}[xx^\top] =\frac{I_{d}}{d},
\end{align}
as well as rotation invariance of the distribution to derive the following bounds.
\begin{lemma}
	\label{lemma:bound exp. norm}
	Let $x\sim \mathrm{Unif}(\mathbb{S}^{d-1})$  then for any matrix $A\in\mathbb{R}^{d\times d}$
	\begin{align}
	\frac{\|A\|_{*}}{d}\leq\mathbb{E}\|Ax\|_{2}\leq \frac{\|A\|_{F}}{\sqrt{d}}.
	\end{align}
\end{lemma}
\begin{proof}
  For the lower bound, we consider the full SVD decomposition of $A=U_{A}\Sigma_AV_{A}^\top$, then by
  Cauchy-Schwartz inequality we get
  \begin{align}
	\label{eq:main estimate convex}
	\E \|Ax\|_2= \E \|\Sigma_{A}V_{A}x\|_2\geq\E \left[x^\top V_{A}^\top \Sigma_{A}V_{A}x\right]=
    \frac{\Tr(\Sigma_{A})}{d}=\frac{\|A\|_{*}}{d}.
  \end{align}
  For the upper bound, we use the concavity of the square root to get
  \begin{align}
	\label{eq:res 3 estimate expectation norm}
	\mathbb{E}\|Ax_1\|_{2}\leq \sqrt{\Tr(A^\top A\  \mathbb{E}[x_1 x_1^\top])} = \frac{\|A\|_F}{\sqrt{d}}.
  \end{align} 
\end{proof}

For quantities depending on $x,y\sim\mathrm{Unif}(\mathbb{S}^{d-1})$ independently, it is handy to
express $y$ in spherical coordinates with respect to $x$. Consider $U_{xy}:=[u_{xy,1},u_{xy,2}]$
such that $U^\top_{xy}U_{xy}=I$ and
\begin{align}
  x=U_{xy}\left[\begin{array}{c}
      1\\0\end{array}\right],\quad y =U_{xy}\left[\begin{array}{c}
      \cos\theta_1\\\sin\theta_1\end{array}\right] ,\quad  0\leq \theta_1 < \pi.
\end{align}
Some examples of this change of coordinates are
\begin{align}
  \label{eq: spherical coordinates}
  \|x-y\|_{2} = 2 \sin\frac{\theta_1}{2},\ \|x+y\|_{2} = 2
  \cos\frac{\theta_1}{2},\ \frac{x-y}{\|x-y||_{2}} = \sin\frac{\theta_1}{2} x +
  \cos\frac{\theta_1}{2}u_{xy,2}
\end{align}
where $u_{xy,2}$ is uniformly distributed on the sphere $\mathbb{S}^{d-2}$ orthogonal to $x$. Now,
to compute expectations, we use the area element of a sphere,
\begin{align}
  \label{eq: area element}
  d_{\mathbb{S}^{d-1}}V = \sin^{d-2}(\theta_1)\sin^{d-3}(\theta_2)\cdots\sin(\theta_{d-2})d\theta_1\cdots d\theta_{d-1}
\end{align}
and the definition of the Beta function,
\begin{align}
  B\left(\frac{m+1}{2},\frac{n+1}{2}\right) :=2\int_{0}^{\pi/2}\sin^{n}\theta\cos^{m}\theta d\theta
  =
  \frac{\Gamma\left(\frac{m+1}{2}\right)\Gamma\left(\frac{n+1}{2}\right)}{\Gamma\left(\frac{m+n}{2}+1\right)}.
\end{align}
where the Gamma function satisfies $\Gamma(n)=(n-1)!$ for any $n\in \mathbb{N}^{+}$. We list some
useful results of expectations relating two independent random variables in $\mathbb{S}^{d-1}$ in
the following lemma.

\begin{lemma}\label{lemma:exp_x_y sphere}
  Let $x,y\sim \mathrm{Unif}(\mathbb{S}^{d-1})$ i.i.d. then
  \begin{subequations}
	\begin{align}
	  \label{eq: useful exp 1}
	  \E \|x-y\|_{2} &= 2\frac{B\left(d-1,\frac{1}{2}\right)}{B\left(\frac{d-1}{2},\frac{1}{2}\right)},\\
	  \label{eq: useful exp 2}
	  P(\|x-y\|_{2}<\delta)&\leq \frac{\delta^d}{5\sqrt{d}}\quad\text{for}\quad \delta < 2,\\
	  \label{eq: useful exp 3}
	  \E\left[ \frac{1}{\|x-y\|_{2}\|x+y\|_{2}}\right]&= \frac{B\left(\frac{d-2}{2},\frac{1}{2}\right)}{2B\left(\frac{d-1}{2},\frac{1}{2}\right)},\\
	  \label{eq: useful exp 4}
	  \E \left[\frac{x-y}{\|x-y\|_{2}} x^\top \right]& =\frac{1}{d} \frac{B\left(d-1,\frac{1}{2}\right)}{B\left(\frac{d-1}{2},\frac{1}{2}\right)}\ I_{d}.
	\end{align}
  \end{subequations}
\end{lemma}

\begin{proof}
  Using the change of variables given by \eqref{eq: spherical coordinates} and the area element
  \eqref{eq: area element}, for \eqref{eq: useful exp 1}, we get
  \begin{align}
	\E \|x-y\|_{2} = 2\frac{\int_{0}^\pi \sin\frac{\theta_{1}}{2}\sin^{d-2}{\theta_1} d\theta_{1}}{\int_{0}^\pi \sin^{d-2}{\theta_1} d\theta_{1}}= 2^{d-1}\frac{B\left(\frac{d}{2},\frac{d-1}{2}\right)}{B\left(\frac{d-1}{2},\frac{1}{2}\right)} = 2\frac{B\left(d-1,\frac{1}{2}\right)}{B\left(\frac{d-1}{2},\frac{1}{2}\right)}  ,
	\end{align}
	using the product identity of the Gamma function $\Gamma(n)\Gamma(n+1/2) = 2^{1-2n}\Gamma(1/2) \Gamma(2n)$.
	
	Using the same change of variables in \eqref{eq: spherical coordinates} for \eqref{eq: useful exp 2}, we get
	\begin{align}
	P(\|x-y\|\leq\delta)
	=P\left(\sin\frac{\theta_{1}}{2}\leq\frac{\delta}{2}\right)
	=2^{d-1}\frac{\int_{0}^{\sin^{-1}(\delta/2)}\sin^{d-1}\theta\cos^{d-2}\theta d\theta}{2\int_{0}^{\pi/2} \sin^{d-2}{\theta} d\theta}.
	\end{align}
	Then, letting $u=\sin^2\theta$, and using $(1-\delta^2/4)<1$ we get,
	\begin{align}
	  P(\|x-y\|\leq\delta)\leq
      2^{d-2}\frac{\int_{0}^{\delta^{2}/4}u^{\frac{d}{2}-1}du}{B\left(\frac{d-1}{2},\frac{1}{2}\right)}=\frac{\delta^{d}}{2d
        B\left(\frac{d-1}{2},\frac{1}{2}\right)}\leq \frac{\delta^{d}}{5 \sqrt d}.
	\end{align}
	
	For \eqref{eq: useful exp 3}, we use the representation on \eqref{eq: spherical coordinates} to get
	\begin{align}
	  \E\left[ \frac{1}{\|x-y\|_{2}\|x+y\|_{2}}\right] = \frac{\int_{0}^\pi
        \sin^{-1}\theta_{1}\sin^{d-2}{\theta_1} d\theta_{1}}{2\int_{0}^\pi \sin^{d-2}{\theta_1}
        d\theta_{1}} =
      \frac{B\left(\frac{d-2}{2},\frac{1}{2}\right)}{2B\left(\frac{d-1}{2},\frac{1}{2}\right)}.
	\end{align}
	
	For \eqref{eq: useful exp 4}, we notice that $\E\left[u_{xy,2}|x\right] = 0$ given that
    $u_{xy,2}$ is uniformly distributed in the sphere orthogonal to $x$. Therefore using \eqref{eq:
      useful exp 1}, we get
	\begin{align}
	\E \left[\frac{x-y}{\|x-y\|_{2}} x^\top \right] 
	=\E \left[\sin\frac{\theta_1}{2} x  x^\top \right] = \frac{1}{d}\frac{B\left(d-1,\frac{1}{2}\right)}{B\left(\frac{d-1}{2},\frac{1}{2}\right)}\ I_{d}.
	\end{align}
\end{proof}

At last, we are interested in inner products of the form
\begin{align}
  \left\langle S x, \frac{x-y}{\|x-y\|_{2}}\right\rangle\quad \text{where}\quad S\in\mathcal{S}_{\mathrm{skew}}(d).
\end{align}
Since $S$ is skew-symmetric, then $x^\top Sx=0$. Therefore if $Sx\neq 0$, we can use $Sx/\|Sx\|_{2}$
as a second orthonormal vector to express $y$ in spherical coordinates. Assuming $d\geq 3$, we
consider
\begin{align}
  U_{xSy}:=[u_{xSy,1},u_{xSy,2},u_{xSy,3}],\quad U^\top_{xSy}U_{xSy}=I 
\end{align}
such that for some $0\leq \theta_1,\theta_2 < \pi$
\begin{align}
  x=U_{xSy,1}\left[\begin{array}{c}
      1\\0\\0\end{array}\right],\quad Sx=\|Sx\|_{2}U_{xSy,2}\left[\begin{array}{c}
      0\\1\\0\end{array}\right],y =U_{xSy,3}\left[\begin{array}{c}
      \cos\theta_1\\\sin\theta_1\cos\theta_{2}\\\sin\theta_{1}\sin\theta_{2}\end{array}\right].
\end{align}
Therefore, under this change of coordinates 
\begin{align}
  \label{eq: spherical coordinates 2}
  \left\langle S x, \frac{x-y}{\|x-y\|_{2}}\right\rangle=\|Sx\|_{2}\cos\frac{\theta_1}{2}\cos{\theta_{2}}.
\end{align}
\begin{lemma}
  \label{lemma: MGF xS(x-y)}
  Let $x,y\in\mathrm{Unif}(\mathbb{S}^{d-1})$ i.i.d and $S\in\mathcal{S}_{skew}(d)$.  If $\|S\|_{2}\leq1$, then 
  \begin{align}
	\label{eq: MGF xS(x-y)}
	\E \left[\exp\left(\lambda \left\langle Sx,
      \frac{x-y}{\|x-y\|_{2}}\right\rangle\right)\right]\leq \exp
    \left(\frac{\lambda^2}{4(d-1)}\right),\quad \lambda>0
  \end{align}
\end{lemma}
\begin{proof}
  To show \eqref{eq: MGF xS(x-y)}, we compute a bound for each of the moments using the change of
  variables in \eqref{eq: spherical coordinates 2} and the area element \eqref{eq: area element},
  \begin{align}
	\label{eq: exp moments r}
	\E  \left\langle Sx,  \frac{x-y}{\|x-y\|_{2}}\right\rangle^r &= \E \left[\|{S}x\|_2^r \cos^r\left(\frac{\theta_1}{2}\right) \cos(\theta_2)^r\right]\nonumber\\ 
	&\leq \E \left[\cos^r\frac{\theta_1}{2}\right]\E \left[ \sin^r\theta_2\right]\nonumber\\ 
	&=  \frac{\int_{0}^{\pi}\cos^{r}\left(\frac{\theta_1}{2}\right)\sin^{d-2}\theta_1\ d\theta_{1}}{\int_{0}^{\pi}\sin^{d-2}\theta_1 \ d\theta_1}\frac{\int_{0}^{\pi}\cos^{r}\theta_2\sin^{d-3}\theta_2\ d\theta_2}{\int_{0}^{\pi}\sin^{d-3}\theta_2\ d\theta_2}.
  \end{align}
  Since $\cos\theta_2$ is an odd function around $\pi/2$, then  \eqref{eq: exp moments r} is 0 for odd values of $r$. For even values of $r$,
  \begin{align}
	\E  \left\langle Sx,  \frac{x-y}{\|x-y\|_{2}}\right\rangle^r &\leq \frac{\int_{0}^{\pi/2}\cos^{r+d-2}\theta_1 \sin^{d-2}\theta_1 d\theta_1}{\int_{0}^{\pi/2} \cos^{d-2}\theta_1\sin^{d-2}\theta_1 d\theta_1 }\frac{\int_{0}^{\pi/2}\cos^{r}\theta_2\sin^{d-3}\theta_2 d\phi}{\int_{0}^{\pi/2}\sin^{d-3}\theta_2 d\theta_2}\nonumber\\ 
	&=\frac{B\left(\frac{r+d-1}{2},\frac{d-1}{2}\right)}{B\left(\frac{d-1}{2},\frac{d-1}{2}\right)}\frac{B\left(\frac{d}{2}-1,\frac{r+1}{2}\right)}{B\left(\frac{d}{2}-1,\frac{1}{2}\right)}\nonumber\\
	&=\frac{\Gamma(d-1)\Gamma\left(\frac{r+1}{2}\right)}{\Gamma\left(\frac{r}{2}+d-1\right)\Gamma\left(\frac{1}{2}\right)}.
  \end{align}
  
  Adding all the moments together and using the product identity of the Gamma function, we get
  \begin{align}
	\E\left[\exp\left(\lambda  \left\langle Sx,  \frac{x-y}{\|x-y\|_{2}}\right\rangle\right)\right] &\leq\sum_{r=0}^\infty \frac{\lambda^{2r}}{\Gamma(2r+1)}\frac{\Gamma\left(\frac{2r+1}{2}\right)}{\Gamma\left(\frac{1}{2}\right)} \frac{\Gamma\left({d-1}\right)}{\Gamma\left(r+d-1\right)}\nonumber\\
	&\leq\sum_{r=0}^\infty \left(\frac{\lambda}{2}\right)^{2r}\frac{1}{\Gamma(r+1)} \frac{\Gamma\left({d-1}\right)}{\Gamma\left(r+d-1\right)}.
  \end{align}
  Then, using the bound
  $\left(\Gamma\left({d-1}\right)\right)\left(\Gamma\left(r+d-1\right)\right)^{-1}<(d-1)^{-r}$, we
  get \eqref{eq: MGF xS(x-y)}.
\end{proof}

\section{Proof of \cref{lemma:bound difference norms}: Bounds for $\|u+v\|_{2}-\|v\|_{2}$}\label{sec:bound difference norms}
Using Taylor series around $\|u\|_2=0$, we construct lower and upper bound for
$\|u+v\|_{2}-\|v\|_{2}$, that are useful in dealing with the contribution of corrupted points to
$L(A)-L(I)$.\newline 

  \begin{proof}[Proof of \cref{lemma:bound difference norms}]
  Let $\alpha\geq 0$ and $\theta\in(0,\pi)$ such that $\alpha = \|u\|_{2}/\|v\|_{2}$ and
  $\cos\theta = (u^\top v)/(\|u\|_{2}\|v\|_{2})$. Then, we can rewrite $\|u+v\|_{2}-\|v\|_{2}$ as
  a function of $\alpha$
  \begin{align}
	f(\alpha)=\|v\|_{2}\left(\sqrt{1+2\alpha\cos\theta+\alpha^2}-1\right).
  \end{align}
  Our goal is to use Taylor expansion of $f(\alpha)$ around $\alpha=0$ and bound the high degree contributions. First note that 
  \begin{align}
	\frac{df}{d\alpha}(\alpha) =\frac{ \|v\|_{2}(\alpha + \cos\theta)}{\left(\sin^2\theta+(\cos\theta+\alpha)^2\right)^2},\quad  \frac{d^2f}{d\alpha^2}(\alpha)= \frac{\|v\|_{2}\sin^2\theta}{\left(\sin^2\theta+(\cos\theta+\alpha)^2\right)^{3/2}}\leq \frac{\|v\|_{2}}{\sin\theta}
  \end{align}
  Therefore, 
  \begin{align}
	\alpha  \|v\|_{2}\cos\theta 	\leq f(\alpha) \leq \alpha  \|v\|_{2}\cos\theta + \frac{\alpha^2}{2}\frac{\|v\|_{2}}{\sin\theta}.
  \end{align}
  Replacing the values of $\alpha$ and $\theta$, we get the bound.
\end{proof}

\section{Proof of \cref{lemma:concentration of residuals}}\label{sec:con_res}

  \begin{proof}[Proof of \cref{lemma:concentration of residuals}]
  Let $X_A$ be
  \begin{equation}\label{eq:Xp in talagrand}
	X_A(\{x_{i}\}_{i=1}^{n}):=\sum_{i = 1}^n f_i(A,x_{i})  - \mathbb{E} f_i(A,x_{i}).
  \end{equation}
  Then, using Hoeffding's inequality (\cref{theorem:hoeffding}), we get
  \begin{multline}
	\|(X_{A}-\mathbb{E}X_{A})-(X_{B}-\mathbb{E}X_{B})\|^2_{\psi_{2}}\leq\\
    c^2\sum_{i = 1}^{n}\|(f_i(A,x_{1})-\mathbb{E}f_i(A,x_{1}))-(f_i(B,x_{1})-\mathbb{E}f_i(B,x_{1}))\|_{\psi_{2}}^2.
  \end{multline}
  
  For any subgaussian random variable $X$, $\|X-\E X\|_{\psi_2}\leq \tilde{c}\|X\|_{\psi_2}$ with
  absolute constant $\tilde{c}$. Then, to satisfy condition \eqref{eq: subgaussian
    increment} gives
  \begin{align}
	\|(X_{A}-\mathbb{E}X_{A})-(X_{B}-\mathbb{E}X_{B})\|^2_{\psi_{2}} \leq c'^2nK\|A-B\|_{F}^{2}.
  \end{align}
  
  To measure the size of $\mathcal{B}_{1}(d):=\{A\in\mathbb{R}^{d\times d}\mid\|A\|_{F}\leq 1 \}$, let $G\in \mathbb{R}^{d\times d}$ be a random
  matrix where each entry distributes independent $\mathcal{N}(0,1)$, then from the
  definitions in \eqref{eq:gauss width rad def}, the Gaussian width of $ \mathcal{B}_{1}(d)$ is
  \begin{equation}
	\label{eq:gaussian width of norm1}
	w\left( \mathcal{B}_{1}(d)\right) = \mathbb{E}\sup_{\|A\|_F\leq 1} \langle A, G\rangle =
    \mathbb{E} \|G\|_F \leq (\mathbb{E} \|G\|_F^2)^{1/2} = d,
  \end{equation}
  and the radius of $\mathcal{B}_{1}(d)$ is $\text{rad}\left( \mathcal{B}_{1}(d)\right) = 1$. Since $X_{A}-\mathbb{E}X_{A}$
  satisfies \eqref{eq:cond_tal}, \cref{theorem:talagrand} implies that
  \begin{equation}
	\mathbb{P}\left\{\sup_{ \|A\|_F\leq 1}  \vert X_A \vert \leq c'K\sqrt{n}(d+u) \right\} \geq 1-2\exp(-u^2).
  \end{equation}
  Taking $u=\sqrt {\log n}/K$ gives the desired bound.
\end{proof}

\section{Proof of \cref{lemma:dyn_SO}: Dynamics of $\|\log R(t)\|_{2}$}\label{sec: app_proof_dyn}
We first introduce the following lemma.
\begin{lemma}[Derivative of matrix logarithm]
  \label{lemma:der. trace. log}
	For $A\in\mathbb{C}^{d\times d}$ with no eigenvalues in $\mathbb{R}^{-}$ 
	\begin{align}
	\frac{d}{dt}\log(A) &= \int_{0}^{1}\left[s(A-I)+I\right]^{-1}\frac{dA}{dt} \left[s(A-I)+I\right]^{-1} ds.	
	\end{align}
\end{lemma}

\begin{proof}
  For any $A\in\mathbb{C}^{d\times d}$ with no eigenvalues in $\mathbb{R}^{-}$, $\log A$ can be defined as \cite{higham_2008}
  \begin{align}
	\log(A)=\int_{0}^{1} (A-I)\left[s(A-I)+I\right]^{-1} ds.
  \end{align}
  Therefore, let $A$ be a function of $t$, then the derivative of $\log A$ with respect to $t$ is
  \begin{align}
	\frac{d\log A}{dt} &= \int_{0}^{1} \frac{dA}{dt}\left[s(A-I)+I\right]^{-1}  - s(A-I)\left[s(A-I)+I\right]^{-1}\frac{dA}{dt} \left[s(A-I)+I\right]^{-1} ds \nonumber\\
	&= \int_{0}^{1}\left[s(A-I)+I\right]^{-1}\frac{dA}{dt} \left[s(A-I)+I\right]^{-1} ds.	
  \end{align}
\end{proof}

  \begin{proof}[Proof of \cref{lemma:dyn_SO}]
  This lemma is tailored from \cite[Lemma 1]{Baccioti_99}. Notice $\|\log(R(t))\|_{2}$ is
  absolutely continuous in any interval $\mathcal{I}\subset \mathbb{R}$ because it is a composition
  of a locally Lipschitz function and an absolutely continuous function. Therefore $\frac{d\|\log
    R(t)\|_{2}}{dt}$ and $\frac{dR(t)}{dt}$ exists almost everywhere and
  \begin{align}\label{eq:def_dr_dt_S}
	\frac{dR(t)}{dt}\in R(t)\cdot \mathcal{S}(R(t))\quad \text{almost everywhere}.
  \end{align}
	
  To define $\|\log R(t)\|_{2}$, we consider the SVD of $\log R(t)$
  \begin{align}
	\label{eq: decomp LogR}
	\log R(t) = \sum_{j=1}^{\left\lfloor d/2 \right\rfloor} \sigma_{j}(t)U_{j}(t) A_{2} U_{j}(t)^\top
  \end{align}
  with
  \[
  U(t) = \left[\begin{array}{ccccc}
	  U_1(t) & \dots & U_{\left\lfloor d/2 \right\rfloor}(t)
	\end{array}\right]
  \]
  such that $\{U_{j}(t)\}_{j=1}^{\left\lfloor d/2 \right\rfloor}\subset\mathbb{R}^{d\times 2}$,
  $U(t)^\top U(t) = I$, and $\pi\geq\sigma_{1}(t)\geq\dots\geq\sigma_{\left\lfloor d/2
    \right\rfloor}(t)\geq 0$. Then
  \begin{align}
	\|\log R(t)\|_{2} := \max\{\sigma_1(t),\cdots,\sigma_{\left\lfloor d/2 \right\rfloor}(t)\}.
  \end{align}
  
 To study $\frac{d\|\log R(t)\|_{2}}{dt}$, we first consider the dynamics of all the
  singular values of $\log R(t)$. Let $\mathrm{blockdiag}(\{V_i\}_{i=1}^{k})$ be the block diagonal matrix
  with diagonal blocks $V_{i}$.  We denote
  \begin{align}
	\Sigma(t) &:= \mathrm{blockdiag}\left(\begin{array}{ccc}
	  \sigma_{1}(t)I_2 ,& \cdots ,&\sigma_{\left\lfloor d/2 \right\rfloor}(t)I_2
	\end{array}\right)
  \end{align}
  and
  \[
  V(t) := \left[\begin{array}{ccccc}U_1(t) A_{2}^\top &\dots & U_{\left\lfloor d/2
        \right\rfloor}(t)A_{2}^\top
	\end{array}\right],
  \]
  then $\log R(t) = U(t) \Sigma(t) V(t)^\top$. 
  
  Recall that for any $A\in\mathbb{R}^{d\times d}$ with SVD decomposition
  $A=U_{A}\Sigma_{A}V_{A}^\top$,
  \begin{align}
	\frac{d\Sigma_{A}}{dt} = I\odot\left(U_{A}^\top\frac{dA}{dt}V_{A}\right)
  \end{align}
  where $\odot$ denotes the entry-wise product. Then, in particular, for $\log R$
  \begin{align}
	\label{eq: der.sing val 1st part}
	\frac{d\Sigma(t)}{dt}
	&=I\odot \left[U(t)^\top\left( \frac{d\log R}{dt}(t) \right)V(t)\right].
  \end{align}
  Since $I$ is a diagonal matrix, then for any matrix $M$, $ I\odot M=I\odot\mathrm{diag}(M)$. In
  particular,
  \begin{align}
	\label{eq: der.sing val 2nd part}
	\frac{d\Sigma}{dt} &=I\odot \mathrm{blockdiag}\left( \left\{U_{j}^\top \left( \frac{d\log R}{dt}
    \right) U_{j}A_2^\top\right\}_{j=1}^{\left\lfloor d/2 \right\rfloor} \right).
  \end{align}
  
  To compute $ \frac{d\log R(t)}{dt}$, we use the fact that $\frac{dR}{dt}(t) = R(t) S(t)$,
  $S(t)\in\mathcal{S}(t)$, then \cref{lemma:der. trace. log} implies that
  \begin{align}\label{eq:def_dlog}
	\frac{d\log R}{dt}
	&= \int_{0}^{1}\left[s(R-I)+I\right]^{-1}R S \left[s(R-I)+I\right]^{-1} ds.
  \end{align}
  To evaluate \eqref{eq:def_dlog}, we rewrite $R(t)$ using the planar decomposition in \cref{lemma:
    planar decomposition R},
  \begin{align}
	\label{eq:plan. dec. R}
	R(t) - I_{d} = \sum_{i=1}^{\left\lfloor d/2
      \right\rfloor}U_{i}(t)(R_{\sigma_{i}(t)}-I_2)U_{i}(t)^\top,\quad R_{\sigma_i}=\exp(\sigma_i
    A_{2})\in\SO(2).
  \end{align}
  Using decomposition \eqref{eq:plan. dec. R}, for all $i = 1,\dots,\left\lfloor d/2 \right\rfloor$,
  we simplify the product
  \begin{align}
	U_{i}^\top \left( \frac{d\log R}{dt} \right)U_{i} &=
    \int_{0}^{1}\left[s(R_{\sigma_{i}}-I)+I\right]^{-1}R_{\sigma_{i}}U_{i}^\top S\ U_{i}
    \left[s(R_{\sigma_{i}}-I)+I\right]^{-1} ds \label{eq: UiLogUi}
  \end{align}
  Notice that $U_{i}^\top S U_{i}\in\mathcal{S}_{\mathrm{skew}}(2)$ then $U_{i}^\top SU_{i} =
  \frac{1}{2} \left\langle S, U_{i} A_{2} U_{i}^\top\right\rangle A_{2}$. Therefore \eqref{eq:
    UiLogUi} becomes
  \begin{align}
	\label{eq: UiLogUi2}
	U_{i}^\top \left( \frac{d\log R}{dt}\right)U_{i} &= \frac{1}{2} \left\langle S, U_{i} A_{2}
    U_{i}^\top\right\rangle \int_{0}^{1}\left[s(R_{\sigma_{i}}-I)+I\right]^{-1}R_{\sigma_{i}}A_{2}
    \left[s(R_{\sigma_{i}}-I)+I\right]^{-1} ds\nonumber\\ &= \frac{1}{2} \left\langle S, U_{i} A_{2}
    U_{i}^\top\right\rangle A_{2} ,
  \end{align}
  given that $A_{2}$ commutes with any rotation in $\SO(2)$.
  
  Inserting \eqref{eq: UiLogUi2} in RHS of \eqref{eq: der.sing val 2nd part} we get
  \begin{align}
	\label{eq: der.sing val 3rd part}
	\frac{d\Sigma}{dt}
	&= \mathrm{blockdiag}\left(\left\{
	\frac{1}{2} \left\langle S(t), U_{j}(t) A_{2} U_{j}(t)^\top\right\rangle I_{2}\right\}_{j=1}^{\left\lfloor d/2 \right\rfloor}
	\right).
  \end{align}
  Using the definition of $\Sigma(t)$, we get
  \begin{align}
	\frac{d\sigma_{i}(t)}{dt}
	&=
	\frac{1}{2} \left\langle S(t), U_{i}(t) A_{2} U_{i}(t)^\top\right\rangle \quad  i=1,\dots, \left\lfloor d/2 \right\rfloor.
  \end{align}
  Therefore, the generalize gradient of $\|\log R\|_{2}$ with respect to $t$ is given by
  \begin{align}
	\label{eq:grad norm Log}
	\frac{d \|\log R(t)\|_{2}}{dt} &= \mathrm{conv}\left\{\frac{1}{2} \left\langle S(t), U_{i}(t)
    A_{2} U_{i}(t)^\top\right\rangle\mid \text{for all } i\ \text{s.t.}\ \|\log R(t)\|_{2}=\sigma_i(t)\right\}.
  \end{align}
  Using the decomposition of $\log R$ in \eqref{eq: decomp LogR}, we have
  \begin{align}
  \{U_{i}(t)\mid
  \ i\ \text{s.t.}\ \|\log R(t)\|_{2}=\sigma_i(t)\} = \mathcal{U}(R(t))
  \end{align}
   as defined in \eqref{eq:def
    UR}. Then, for all $t$ such that $\frac{d\|\log R(t)\|_{2}}{dt}$ exists, then RHS of
  \eqref{eq:grad norm Log} is a singleton. Therefore
  \begin{align}
	\frac{d\|\log R(t)\|_{2}}{dt} \in \{a\mid \exists S \in\mathcal{S}(t), \langle S,UA_{2}U^\top
    \rangle = 2a\ \forall U\in \mathcal{U}(R(t))\}\
  \end{align}
  almost everywhere.
\end{proof}

\section{Proof of \cref{lemma:cover SO}: Constructing $\varepsilon$-net for $\SO(d)$}\label{sec: app_enet} 
We first introduce the following theorem.
\begin{lemma}[Lipschitz constant of the matrix exponential in $\mathcal{S}_{\mathrm{skew}}(d)$]
  \label{lemma: matrix exp}
  Let $X,Y$ skew-symmetric matrices, then 
  \begin{align}
	\|\exp(X)-\exp(Y)\|_{F}\leq \left\|X-Y\right\|_{F}.
  \end{align}
\end{lemma}
\begin{proof}
  Let the directional derivative of matrix exponential of $Z$ in direction $X-Y$ be \cite{higham_2008}
  \begin{align}
	d\exp_{Z}(X-Y) = \int_{0}^{1}\exp((1-t)Z)(X-Y)\exp(tZ)\ dt.
  \end{align}
  By continuity of the matrix exponential, we have that for $X,Y\in \mathcal{S}_{\mathrm{skew}}(d)$,
  there exist $Z\in\mathcal{S}_{\mathrm{skew}}(d)$ such that
  \begin{align}
	\|\exp(X)-\exp(Y)\|_{F}&\leq \left\|\int_{0}^{1}\exp((1-t)Z)(X-Y)\exp(tZ)\ dt\right\|_{F},
  \end{align}
  given that  $\mathcal{S}_{\mathrm{skew}}(d)$ is a vector space. By triangle inequality,
  \begin{align}
	\|\exp(X)-\exp(Y)\|_{F}&\leq\int_{0}^{1} \left\|\exp((1-t)Z)(X-Y)\exp(tZ)\right\|_{F}\ dt = \|X-Y\|_{F},
  \end{align}
  where the last equality follows from the fact that the exponential of a skew-symmetric matrix is a
  rotation, and the Frobenius norm is invariant under rotations.
\end{proof}

 \begin{proof}[Proof of \cref{lemma:cover SO}]
  Recall $\mathcal{A} := \{S\in \mathbb{R}^{d\times d} \ \mid \ \|\mathrm{skew}(S)\|_2 \leq
  1/\sqrt{2},\ S(i,j)=0\ \text{if}\ i\geq j \}$.  Let $\mathrm{triu}(\cdot):\mathbb{R}^{d\times
    d}\rightarrow\mathbb{R}^{d\times d}$ be defined as
  \begin{align}
	\mathrm{triu}(A)(i,j) := \begin{cases}
	  A(i,j) & \text{if}\quad i\leq j\\
	  0 & \text{otherwise}.
	\end{cases}
  \end{align}
  Notice that for any $S\in \mathcal{B}_{\mathrm{skew}}(d)$, $\mathrm{triu}(2^{-1/2}S)
  \in\mathcal{A}$. Therefore, there exists $\tilde{S}\in\mathcal{N}^{\varepsilon}_{\mathcal{A}}$
  such that $\left\|\text{triu}\left(2^{-1/2} S\right)-\tilde{S}\right\|_{F}\leq \varepsilon$. Since
  $\|\mathrm{skew}(\sqrt{2}\tilde{S})\|_{2}\leq1$ then
  \begin{align}
	\label{eq: enet skew}
	\left\|S - \mathrm{skew}({\sqrt{2}}\ \tilde{S})\right\|_{F}=\left\|\text{triu}\left(2^{-1/2}S\right)-\tilde{S}\right\|_{F}\leq\varepsilon.
  \end{align}
  Similarly, for $R\in \SO(d)$, let $S_{R} = \log(R)/\pi$. Then $R=\exp(\pi S_{R})$ and
  $S_{R}\in\mathcal{B}_{\mathrm{skew}}(d)$. Therefore, there exists $\tilde{S}_{R}\in\mathcal{N}_{\mathcal{A}}^{\varepsilon}$
  such that $\|S_R-\mathrm{skew}(\sqrt{2}\tilde{S}_{R})\|_{F}\leq \varepsilon$. Then, by \cref{lemma: matrix
    exp}, 
  \begin{align}
	\|R-\exp(\pi\mathrm{skew}(\sqrt{2}\ \tilde{S}_{R}))\|_{F} \leq \|\pi S_R -
    \pi\mathrm{skew}(\sqrt{2}\tilde{S}_{R}))\|_{F}\leq \pi\ \varepsilon.
  \end{align}
		
  We can also provide an upper bound to the size of $\mathcal{N}^{\varepsilon}_{\mathcal{B}}$ and
  $\mathcal{N}^{\pi\varepsilon}_{\SO}$ given the size of
  $\mathcal{N}^{\varepsilon}_{\mathcal{A}}$. To estimate the size of
  $\mathcal{N}^{\varepsilon}_{\mathcal{A}}$, we consider
  \begin{align}
    \mathcal{A}_{F} := \{S\in \mathbb{R}^{d\times d} \ \vert\  \|S\|_F \leq \sqrt{d},\ S(i,j)=0\ \text{if}\ i\leq j  \}.
  \end{align}
  Let $\mathcal{N}^{\varepsilon/2}_{\mathcal{A}_F}$ be the smallest $\varepsilon/2$-net of
  $\mathcal{A}_{F}$. Since $\mathcal{A}\subset\mathcal{A}_{F}$, \cref{thm:enet monotone} implies that,
  for $\mathcal{N}^{\varepsilon}_{\mathcal{A}}$ the smallest $\varepsilon$-net of $\mathcal{A}$,
  $|\mathcal{N}^{\varepsilon}_{\mathcal{A}}|\leq|\mathcal{N}^{\varepsilon/2}_{\mathcal{A}_F}|$. Now,
  to know the size of $|\mathcal{N}^{\varepsilon/2}_{\mathcal{A}_F}|$, notice that $\mathcal{A}_{F}$
  is an Euclidean ball of radius $\sqrt{d}$ in a vector space of dimension $d(d-1)/2$. Then,
  \cref{thm:enet} implies that
  \begin{align}
    |\mathcal{N}^{\pi \varepsilon}_{\SO}|\leq |\mathcal{N}^{\varepsilon}_{\mathcal{B}}|\leq
    |\mathcal{N}^{\varepsilon}_{\mathcal{A}}|\leq \vert \mathcal{N}^{\varepsilon/2}_{\mathcal{A}_{F}}
    \vert \leq \left(6 \sqrt{d}\ \varepsilon^{-1}\right)^{\frac{d(d-1)}{2}}.
  \end{align}
\end{proof}

\bibliographystyle{siamplain}
\bibliography{biblio}

\end{document}